\documentclass[11pt,a4paper,oneside]{article}
\usepackage{amsfonts, amsmath, amssymb, amsthm, mathtools}
\usepackage[colorlinks=true,citecolor=blue!50!black,linkcolor=red!50!black]{hyperref}
\usepackage[margin=1in]{geometry}
\usepackage[utopia]{mathdesign}
\usepackage[mathscr]{euscript}
\usepackage[utf8]{inputenc}
\usepackage{placeins}

\usepackage{graphicx}
\usepackage{parskip}
\usepackage{enumerate}
\usepackage{verbatim}
\usepackage{tikz}
\usepackage{units}

\usetikzlibrary{calc}
\usepackage[outline]{contour}
\contourlength{1.2pt}

\begingroup
    \makeatletter
    \@for\theoremstyle:=definition,remark,plain\do{%
        \expandafter\g@addto@macro\csname th@\theoremstyle\endcsname{%
            \addtolength\thm@preskip\parskip
            }%
        }
\endgroup

\newtheorem{theorem}{Theorem}[section]
\newtheorem*{theorem*}{Theorem}
\newtheorem{lemma}[theorem]{Lemma}

\theoremstyle{definition}

\newtheorem{remark}[theorem]{Remark}
\newtheorem*{remark*}{Remark}

\newcommand{\df}[1]{{{\color{blue!50!black}\em #1}}}
\newcommand{\mb}[1]{\mathbb{#1}}
\newcommand{\mc}[1]{\mathscr{#1}}

\newcommand{\cm}{,\allowbreak }

\DeclareMathOperator{\ct}{ct}
\newcommand{\fullrs}{{\mc{R}}}
\newcommand{\projrs}{\widetilde{\mc{R}}}
\newcommand{\vcrs}{{\mc{V}}}
\newcommand{\wcrs}{{\mc{W}}}

\newcommand{\indx}{{\mathfrak{L}}}
\newcommand*{\sym}{{\mathfrak{S}}}
\newcommand{\coty}{\Omega}

\newcommand*{\projeq}{%
\mathrel{\vcenter{\offinterlineskip
\hbox{\hskip+1pt{\tiny{\rm proj}}}\vskip+1pt\hbox{\scalebox{2}[1]{$\sim$}}}}}

\newcommand*{\swapeq}{%
\mathrel{\vcenter{\offinterlineskip
\hbox{{\tiny{\rm swap}}}\vskip+1pt\hbox{\scalebox{2}[1]{$\sim$}}}}}

\DeclareMathOperator{\conv}{conv}
\DeclareMathOperator{\bump}{bump}
\DeclareMathOperator{\projrsvc}{sc}
\DeclareMathOperator{\cyclecat}{\circlearrowleft}
\DeclareMathOperator{\vertflip}{\updownarrow}

\newcommand{\eps}{\varepsilon}

\newcommand{\fig}{{Figure}}

\newcommand{\us}{ .. controls +(0.5,0) and +(-0.2,-0.2) .. ++(1,0.5) }
\newcommand{\ue}{ .. controls +(0.2,0.2) and +(-0.5,0) .. ++(1,0.5)}
\newcommand{\ds}{ .. controls +(0.5,0) and +(-0.2,0.2) .. ++(1,-0.5)}
\newcommand{\de}{ .. controls +(0.2,-0.2) and +(-0.5,0) .. ++(1,-0.5)}
\newcommand{\du}{ .. controls +(0.3,-0.3) and +(-0.3,-0.3) .. ++(1,0)}
\newcommand{\ud}{ .. controls +(0.3,0.3) and +(-0.3,0.3) .. ++(1,0)}

\newcommand{\g}[1]{ -- ++(#1,0)} 
\newcommand{\ug}[1]{ -- ++(#1,#1)} 
\newcommand{\dg}[1]{ -- ++(#1,-#1)} 

\newcommand{\usx}{ .. controls +(0.3,0) and +(-0.1,-0.1) .. ++(.75,0.5) }
\newcommand{\uex}{ .. controls +(0.1,0.1) and +(-0.3,0) .. ++(.75,0.5)}
\newcommand{\dsx}{ .. controls +(0.3,0) and +(-0.1,0.1) .. ++(.75,-0.5)}
\newcommand{\dex}{ .. controls +(0.1,-0.1) and +(-0.3,0) .. ++(.75,-0.5)}

\newcommand{\ft}{\footnotesize}

\begin{document}

\title{Realization spaces of arrangements of convex bodies}

\author{
Michael Gene Dobbins 
\thanks{GAIA, POSTECH, South Korea}
\and 
Andreas Holmsen 
\thanks{Department of Mathematical Sciences, KAIST, South Korea}
\and 
Alfredo Hubard
\thanks{GEOMETRICA, Inria Sophia-Antipolis, France}
}

\date{}

\maketitle

\begin{abstract}
We introduce \emph{combinatorial types} of arrangements of convex bodies, extending order types of point sets to arrangements of convex bodies, and study their realization spaces. 
Our main results witness a trade-off between the combinatorial complexity of the bodies and the topological complexity of their realization space. 
First, we show that every combinatorial type is realizable and its realization space is contractible under mild assumptions. 
Second, we prove a universality theorem that says the restriction of the realization space to arrangements polygons with a bounded number of vertices can have the homotopy type of any primary semialgebraic set.\footnote{%
M.~G.~Dobbins was supported by NRF grant 2011-0030044 (SRC-GAIA) funded by the government of South Korea.
A.~Holmsen was supported by Basic Science Research Program through the National Research Foundation of Korea funded by the Ministry of Education, Science and Technology (NRF-2010-0021048). 
A.~Hubard was supported by Fondation Sciences Math\'{e}matiques de Paris and by the Advanced Grant of the European Research Council GUDHI (Geometric Understanding in Higher Dimensions). 
}  
\end{abstract}

\section{Introduction}

We introduce a generalization of order type that provides a framework to study arrangements of convex sets and their convex dependencies.  The notion we introduce is closely related to wiring diagrams \cite{goodburr} or primitive sorting networks \cite{knuti}.  It is also related to double pseudoline arrangements introduced by Pocchiola and Habert \cite{habert} and double allowable sequences introduced by Goodman and Pollack \cite{GPdoub}.  
These related notions have applications in the study of visibility, transversal, and separation properties of convex sets \cite{ragga, novick1, novick2, hubsuk}.  
Combinatorial type, the generalization of order type studied here, was fundamental to the authors' work on generalizations of the Erdős-Szekeres theorem to arrangements of convex sets in the plane \cite{dhh1, dhh2}.  In this paper, we address the relevant realizability questions.

Two indexed point sets $P=\{p_1,p_2\ldots p_n\}$ and $Q=\{q_1,q_2\ldots q_n\}$ in the plane are said to have the same order type when for every triple $(i,j,k)$, the orientation of the triples $p_i,p_j,p_k$ and $q_i,q_j,q_k$ coincides.  
Equivalently, a point set $P$ corresponds to a dual family $P^*$ of oriented great circles in the sphere by projective duality, and point sets $P$ and $Q$ have the same order type when the families $P^*$ and $Q^*$ subdivide the sphere in the same way.  That is, when there is a self-homeomorphism of the sphere that sends each cell of $P^*$ to a corresponding cell of $Q^*$ and preserves orientations. 

More generally, we say that a sign function $\chi : \indx^3 \to \{+,0,-\}$ is an \df{order type} when it satisfies the axioms of rank 3 acyclic chirotopes \cite[page 126]{OMS} \cite[Chapter 10]{knuti}.  
Specifically, $\chi$ is alternating, satisfies the Grassman-Plücker Relations, is acyclic (a restatement of Radon's Partition Theorem in terms of orientations), and is not identically zero. 
Order types that satisfy $\chi( i, j, k) \neq 0$ for any $ i, j, k$ distinct are called \df{simple}, and are equivalent to {\em uniform} rank 3 acyclic chirotopes and to Donald Knuth's CC-systems.
Many geometric properties of point sets (such as convex dependency, transversal, and separation properties) depend solely on the order type of the points and not on the actual coordinates, and these axioms have been used to prove theorems and design algorithms involving such properties \cite{OMS, knuti}.

Not every order type can be realized by points in the plane, or dually by oriented great circles in the sphere.
However, by the Folkman-Lawrence Representation Theorem, any order type can be realized by a pseudocircle arrangement \cite{folkman}.  That is in the case of simple order types, by a family of simple oriented closed curves on the sphere such that each pair of curves intersect at exactly two points, any other curve separates these two points, and some point on the sphere is to the left of every curve.
We may alternatively define a simple order type to be any subdivision of the sphere by such a pseudocircle arrangement $\mc S = \{S_1,\dots,S_n\}$.
The orientations $\chi(i,j,k)$ can be recovered from this subdivision by the order $S_i,S_j,S_k$ appear on the boundary of the region to the left of all three pseudocircles.

Like simple order types, combinatorial types are finite combinatorial objects that can be associated to families of geometric objects, namely arrangements of convex bodies, which are assumed to satisfy certain genericity conditions.
These will be defined precisely in Section \ref{prelims}, but for now we describe the equivalence relation that combinatorial types induce on arrangements of convex bodies. 
To do so, we define a duality for convex bodies that is analogous to projective duality for points in the plane.
The \df{dual support curve} $A^*$ of a convex body $A$ in the plane, is the graph of its support function $h_A : \mb{S}^1 \to \mb{R}^1$, $h_A(\theta) := \max_{p \in A}\langle \theta, p \rangle$ on the cylinder $\mb S^1\times \mb R^1$,
where $\mb{S}^1$ is the unit circle and $\langle\mathbin{\cdot},\mathbin{\cdot}\rangle$ is the standard inner product. 
In this way, every arrangement $\mc{A} = \{A_1,\dots,A_n\}$ corresponds to  
the \df{dual support system} $\mc{ A}^* = \{A_1^*,\dots,A_n^*\}$ of curves on the cylinder $\mb S^1\times \mb R^1$ given by the graphs of the functions $\{h_{A_1},h_{A_2}, \ldots h_{A_n}\}$. 
In the other direction, not all functions $h\colon \mb{S}^1 \to \mb{R}^1$ are support functions, but we do have the following sufficient conditions. 

\begin{remark}\label{blashke}
Blashke showed that if $h\colon \mb{S}^1 \to \mb{R}^1$ is $C^2$-smooth and $h+h''>0$, then $h$ is the support function of a planar curve with curvature bounded by $\frac{1}{h+h''}$ \cite[Lemma 2.2.3]{Groemer1996}. 
Hence, by adding a sufficiently large constant to a family of smooth functions, we can ensure the family is the dual support system of an arrangement of convex bodies.
\end{remark}

\begin{figure}[h]
\centering

\begin{tikzpicture}[scale=1.8]

\draw[thin,black!40]
(45:{sin(25)/sin(35)} ) coordinate (a) +(164:-.2) -- +(164:.7)
 +(106:.2) -- +(106:-.7)
 +(190:1.7) -- +(190:-.2)
(45:.9) -- (45:-.2)
(165:1) +(205:.2) -- +(205:-1.1)
(205:1) +(165:.2) -- +(165:-1.1)
;

\draw[->,>=latex,black!30]
(0,0) -- (315:1.1)
;
\draw[->,>=latex,black!30]
(255:{cos(50)} ) -- (255:1.1)
;
\draw[->,>=latex,black!30]
(115:{cos(50)} ) -- (115:1.1)
;
\draw[->,>=latex,black!30]
(100:{cos(-65)} ) -- (100:1.1)
;
\draw[->,>=latex,black!30]
(74:{cos(50)} ) -- (74:1.1)
;
\draw[->,>=latex,black!30]
(16:{cos(50)} ) -- (16:1.1)
;

\path
(16:1.2) node[anchor=center] {\ft $\theta_1$}
(74:1.2) node[anchor=center] {\ft $\theta_2$}
(100:1.2) node[anchor=center] {\ft $\theta_3$}
(115:1.2) node[anchor=center] {\ft $\theta_4$}
(255:1.2) node[anchor=center] {\ft $\theta_5$}
(315:1.2) node[anchor=center] {\ft $\theta_6$}
;

\draw[thick, blue, fill=blue, fill opacity=.05]
(0,0) -- (165:1) -- (205:1) -- cycle
;

\draw[thick, cyan, fill=cyan, fill opacity=.05]
(0,0) circle ({cos(50)})
;

\node[blue] at (185:.8) {3};
\node[cyan] at (.25,-.03) {1};

\fill[ultra thick,violet]
(a) circle (.7pt)
+(.03,.15) node {2}
;

\end{tikzpicture}

\begin{tikzpicture}[xscale=.03,yscale=1]

\draw[thin,black!30]
(315,-1) -- (315,0)
(255,-1) -- (255,{cos(50)})
(115,-1) -- (115,{cos(50)})
(100,-1) -- (100,{cos(-65)})
(74,-1) -- (74,{cos(50)})
(16,-1) -- (16,{cos(50)})
;

\path
(0,-1.3) node[anchor=base] {\ft $0$}
(16,-1.3) node[anchor=base] {\ft $\theta_1$}
(74,-1.3) node[anchor=base] {\ft $\theta_2$}
(100,-1.3) node[anchor=base] {\ft $\theta_3$}
(115,-1.3) node[anchor=base] {\ft $\theta_4$}
(255,-1.3) node[anchor=base] {\ft $\theta_5$}
(315,-1.3) node[anchor=base] {\ft $\theta_6$}
(360,-1.3) node[anchor=base] {\ft $2\pi$}
;

\draw[thick, blue]
(0,0) node[left] {\ft 3} -- (75,0)
\foreach \t in {80,85,...,185}
{ -- (\t,{cos(\t-165)}) }
\foreach \t in {190,195,...,295}
{ -- (\t,{cos(\t-205)}) }
 -- (360,0)
;

\draw[thick, cyan]
(0,{cos(50)}) node[left,yshift=2pt] {\ft 1} -- (360,{cos(50)})
;

\draw[thick, violet]
(0,{(sin(25)/sin(35))*cos(-45)}) node[left,yshift=-2pt] {\ft 2}
\foreach \t in {5,10,...,360}
{ -- (\t,{(sin(25)/sin(35))*cos(\t-45)}) }
;

\end{tikzpicture}

\caption{\ft 
\textbf{Top:} An arrangement $\mc A$ and its common supporting tangents. 
\textbf{Bottom:} Its dual support system $\mc A^*$. 
}
\label{eg_dual}
\end{figure}
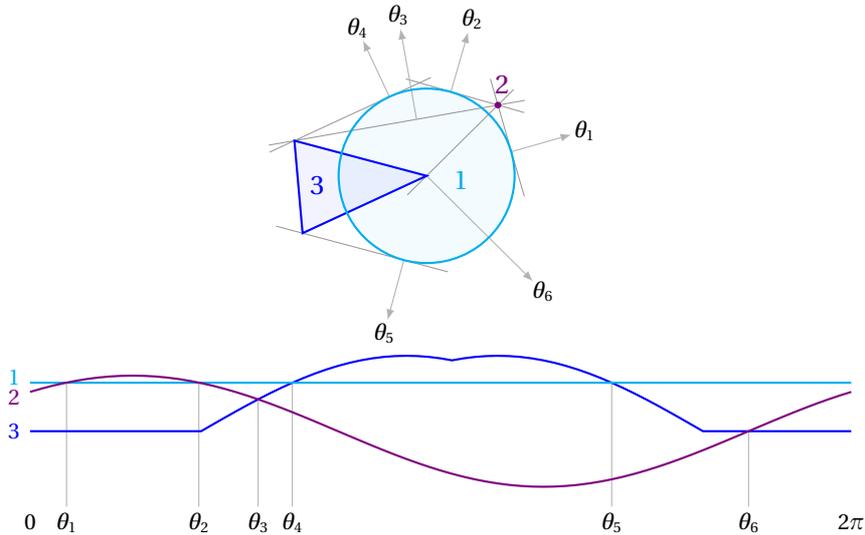

The combinatorial type of an arrangement of bodies $\ct(\mc A)$ is, essentially, a combinatorial encoding of the subdivision of the cylinder $\mb S^1\times \mb R^1$ by the dual support curves $\mc A^*$. 
For now, we take the following theorem as an alternative topological definition.

\vbox{
\begin{theorem}\label{combequivtype}
Two arrangements of convex bodies $\mc A$ and $\mc B$ have the same combinatorial type if and only if their dual systems $\mc{A}^*$ and $\mc{B}^*$ are related by a self-homeomorphism of the cylinder that preserves orientation and $+\infty$. 
\end{theorem}
}

Here, we say that a self-homeomorphism $\phi : \mb{S}^1 \times \mb{R}^1 \to \mb{S}^1 \times \mb{R}^1$ preserves $+\infty$ when for $y$ sufficiently large the second coordinate of $\phi(\theta, y)$ is positive for all $\theta$.  
Equivalently, if $C \subset \mb{S}^1 \times \mb{R}^1$ is an oriented curve that wraps once around the cylinder in the counter-clockwise direction, then so is its image $\phi(C)$.

In the case of points, the duality that we defined through support functions is the usual projective duality renormalized to be on the cylinder.  
Consequently, two generic point sets have the same order type if and only if they have the same combinatorial type. 
Specifically, a point in the plane can be represented in homogeneous coordinates by a line in $\mb{R}^3$, 
and its dual support curve is the intersection of the orthogonal complement of this line with the cylinder embedded in $\mb{R}^3$. 
The same relationship holds between a body in the plane represented by a cone in $\mb R^3$ and the body's dual support curve represented by its polar cone.

\paragraph{Realizing order types.}

Although combinatorial types of arrangements are more general than simple order types, we associate an order type to the class of arrangements defined by the following local condition on triples of bodies. 
We say a triple of bodies is \df{orientable} when it has the combinatorial type of three generic points, 
and we say an arrangement is orientable when it consists of at least three bodies and every triple is orientable.
In the case of an orientable arrangement $\mc{A}$, every triple $\{A_{ i},A_{ j},A_{ k}\}\subset \mc A$ contributes a single connected arc to the boundary of $\conv(A_{ i},A_{ j},A_{ k})$, and we define the \df{orientation} of an ordered triple $(A_{ i},A_{ j},A_{ k})$ to be positive when the bodies appear counter-clockwise in the given order on the boundary, and to be negative otherwise. 
Grünbaum implicitly observes that the cyclic orderings of the triples of an orientable arrangement form an order type in his discussion on planar arrangements of simple closed curves \cite[Section 3.3]{grunbaumS}.

\begin{figure}[h]
\centering

\begin{tikzpicture}

\begin{scope}[shift={(-3,0)}]
\fill[cyan]
(30:.8) circle (2pt)
;
\fill[violet]
(150:.8) circle (2pt)
;
\fill[blue]
(270:.8) circle (2pt)
;
\end{scope}

\begin{scope}[yscale=.5,xscale=.015]

\draw[thick, violet]
(0,{cos(-150)})
\foreach \t in {5,10,...,360}
{ -- (\t,{cos(\t-150)}) }
;

\draw[thick, blue]
(0,{cos(-270)})
\foreach \t in {5,10,...,360}
{ -- (\t,{cos(\t-270)}) }
;

\draw[thick, cyan]
(0,{cos(-30)})
\foreach \t in {5,10,...,360}
{ -- (\t,{cos(\t-30)}) }
;

\path
(0,-2) node (a) {\ft 0} 
(360,-2) node (b) {\ft $2\pi$}
;
\end{scope}

\end{tikzpicture}
\caption{\ft Three generic points and their support curves.}  \label{3points}
\end{figure}
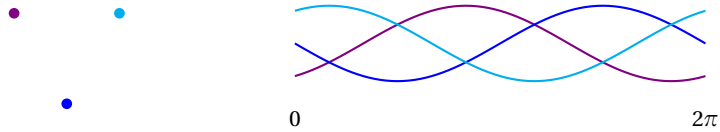

Most order types cannot be realized by points, and in fact, it is NP-hard to decide if a given order type can \cite{shor}. 
Having a notion of combinatorial type allows us to approach questions regarding realizability by bodies rather than points \cite{alfie}.  
The smallest \emph{non-realizable} order type is the Non-Pappus Configuration, a configuration of $9$ elements that violates Pappus's Theorem \cite{levi, ringel2}. 
Pach and T\'{o}th showed that the Non-Pappus Configuration can be realized by an arrangement of segments in the plane \cite{pt3}. 
\fig~\ref{badpenta} shows a non-realizable order type that can be realized by triangles, Goodman and Pollack's ``Bad Pentagon'' \cite{GPallow}, 
and the authors conjecture that this order type cannot be realized by segments. 
In contrast to point sets, we show that any order type, and in fact any combinatorial type, can be realized by an arrangement of bodies.

\begin{figure}[h!] \centering  
\begin{tikzpicture}[scale=.8]
\begin{scope}[scale= -.8]
\clip (0,0) circle (3.2);
\foreach \x in {0,72,...,288}
{
\draw[thin,black!70]
 ({\x-72}:1) .. controls +({\x-18}:1) and +(0,0) .. ({\x-15}:3.2)
 ({\x-72}:1) -- + +({\x-18}:-4)
 (\x:1) .. controls +({\x-18}:1) and +(0,0) .. ({\x-27}:3.2)
 (\x:1) -- +({\x-18}:-5);
}
\foreach \x in {0,72,...,288}
{
\fill[blue]
  (\x:1) circle (2.2pt);
\fill[blue]
  ({\x-21}:3) circle (2.2pt);
}
\end{scope}

\begin{scope}[scale=-0.4, xshift = -20cm]
\foreach \x in {0,...,4}
{\draw[fill=blue!25]
 ({72*\x+27}:0.9) +({72*\x}:-0.4) coordinate (c\x) -- +({72*\x}:0.4)
 coordinate (b\x) -- +({72*\x+90}:-1.2) coordinate (a\x) -- cycle; 
\draw[fill=blue!25]
 ({72*\x-17.5}:8) +({72*\x}:-0.4) -- +({72*\x}:0.4) --
 +({72*\x+90}:-1.2) -- cycle; }
\foreach \i in {0,...,4} 
{\draw[thin, opacity=.2] 
 let \n1={int(mod(\i+2,5))} 
 in ($(b\i)!4!(a\n1)$) -- ($(a\n1)!6!(b\i)$);
\draw[thin, opacity=.2]
 let \n1={int(mod(\i+1,5))} 
 in ($(c\n1)!4!(a\i)$) -- ($(a\i)!6!(c\n1)$);
}
\end{scope}
\end{tikzpicture}
\caption{\ft Two realizations of the Bad Pentagon. \textbf{Left:}
a realization in a topological plane \cite{GPallow}. \textbf{Right:} a realization by
convex sets in the Euclidean plane.} \label{badpenta}
\end{figure}
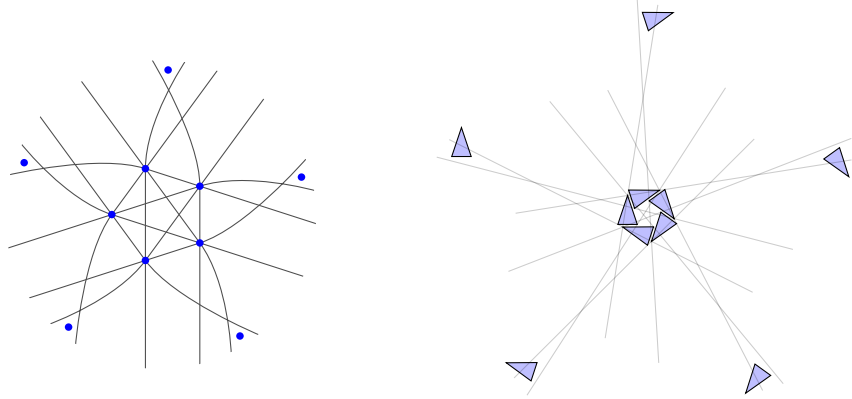

\vbox{
\begin{theorem} \label{orientable} 
The orientations of the triples of any orientable arrangement is a simple order type. 
Two orientable arrangements have the same order type if and only if they have the same combinatorial type.
And, every simple order type can be realized by an orientable arrangement. 
\end{theorem}
}


If we bound the complexity of the bodies, then some simple order types can no longer be realized. 
Indeed, we show that simple order types can always be realized by $k$-gons, but may require $k$ to be arbitrarily large. 

\begin{theorem}\label{k-bounds} 
Let $k_n$ be the smallest integer for which every simple order type on $n$ elements can be realized by an arrangement of $k_n$-gons.
There are constants $c_1,c_2>0$ such that 
\[ c_1 \tfrac{n}{\log n} \leq k_n \leq c_2 n^2. \]
\end{theorem}


\paragraph{Realization spaces.}

An old conjecture of Ringel claimed that given two point sets with the same order type, one point set can be continuously deformed to the other while maintaining the order type \cite{ringel2}. 
This naturally leads to the study of {\em realization spaces} of order types, the set of all families of points with a fixed order type modulo projectivities. 
The conjecture can then be restated as, any non-empty realization space is connected. 
Ringel's conjecture was disproved in the early eighties, and 
the strongest result in this direction is Mnev's Universality Theorem \cite{mnev1, mnev2}, which states
that for any primary semialgebraic set $\mc Z$, 
there exists an order type whose realization space is homotopy equivalent to $\mc Z$. 
Recall that a primary semialgebraic set is the set of common solutions to some finite list of polynomial equations and strict inequalities in several real variables.
This has lead to a growing body of work 
\cite{OMS, milson, padrol2014delaunay, gebertReal, tsukamoto2013new, vakil}.

The main objective of this paper is to extend the study of realization spaces to arrangements of bodies of a fixed combinatorial type and exhibit a trade-off between the combinatorial complexity of the bodies and the topological complexity of their realization space. 
The first indication of this trade-off may be observed from Theorems~\ref{orientable} and~\ref{k-bounds}, which imply that for general convex bodies the realization space of any simple order type is non-empty, but this fails for $k$-gons. 
We prove two contrasting results. First, we show in Theorem~\ref{realizations} that Ringel's intuition is correct in this generalized context: the realization space of any combinatorial type satisfying some mild assumptions is contractible; that is, it is non-empty and has no holes of any dimension.  In particular, the set of arrangements (modulo planar rotations) of convex bodies realizing any fixed simple order type is contractible, and therefore connected. 
Second, we show in Theorem~\ref{kgonrs} that if the bodies are restricted to polygons with at most $k$ vertices, then Mnev's Theorem generalizes.\footnote{Note that Mnev's Theorem is more specific as it
deals with stable equivalence.}  Specifically, we show that for every $k$ and every primary semialgebraic set $\mc Z$ there is a combinatorial type whose $k$-gon realization space is homotopy equivalent to $\mc Z$. 

The proof of Theorem~\ref{realizations} provides an explicit deformation retraction to a standard arrangement, which will be defined for each combinatorial type.  The proof of Theorem~\ref{kgonrs} depends on Mnev's Theorem.  Specifically, the proof uses a reduction from the case of $k$-gons to the case of points by constructing an arrangement of $k$-gons where the only obstruction to realizability is that certain vertices must always have the same fixed order type.

\paragraph{Relationship to double pseudoline arrangements.} \label{2remarks}

Pocchiola and Habert introduced an extension of chirotopes to arrangements of convex sets based on a similar notion of duality to what is presented here, called double pseudoline arrangements \cite{habert}.  The essential difference is that the dual double pseudoline of a convex set is defined as the quotient of the dual support curve by the $\mb{Z}_2$ action on the cylinder $(\theta,y) \sim (\theta,-y)$.  Instead of a curve that wraps monotonically once around the cylinder, the dual double pseudoline is a curve that wraps monotonically twice around the Möbius strip. 
This leads to an extended notation of chirotopes that provides information about arrangements of convex sets which combinatorial types do not distinguish, such as disjointness and visibility. 
On the other hand, combinatorial types distinguish convex position of subarrangements and are simpler in certain respects that are crucial to the analysis in \cite{dhh1,dhh2} and the results of this paper.

\paragraph{Organization of the paper.}
Section \ref{prelims} gives definitions, states the main theorems of the paper, and clarifies some issues that were treated vaguely in the introduction. 
Section \ref{counting} deals with realizing order types, Theorems~\ref{orientable} and~\ref{k-bounds}; this section conveys the overall theme of the paper, but in a less technical setting, and involves ideas similar to what will be used in the rest of the paper.  
Section \ref{contract} proves contractibility, Theorem~\ref{realizations}.  
Section \ref{topinv} proves the topological invariance of combinatorial type, Theorem~\ref{combequivtype}.  While Theorem~\ref{combequivtype} is a more fundamental result, it appears later in the text as it depends crucially on Theorem~\ref{realizations}.  
Section \ref{universe} gives the universality construction for convex $k$-gons, Theorem~\ref{kgonrs}.  Finally, Section \ref{conclude} ends with some remarks and open problems.

\section{Preliminaries and main theorems} \label{prelims}

\paragraph{Genericity assumptions.} \label{sect:generic} 

A \df{common supporting tangent} of a pair of 
bodies is a {\em directed line} tangent to each body such that both
bodies are on its {\em left} side. In the dual, this corresponds to an intersection between two support curves.
We say that a pair of bodies intersect \df{transversally} when no point of intersection is contained in a common supporting tangent. In the dual this corresponds to a pair of curves in the cylinder that cross at each point of intersection; that is, for a pair of curves that are respectively the graphs of functions $f_1, f_2$, the function $f_1-f_2$ has only isolated zeros and changes sign at each zero. 
By an \df{arrangement} we mean a {\em finite} {\em indexed} non-empty collection of compact convex sets, which we call {\em bodies}, that  satisfy following genericity conditions: 
\begin{itemize} 
\item Each pair of bodies intersect transversally.
\item No three bodies share a common supporting tangent.
\item There are finitely many common supporting tangents. 
\end{itemize}
Analogously, by a \df{system} we mean a finite indexed collection of closed curves in the cylinder $\mb{S}^1 \times \mb{R}^1$ that are monotonic in the first coordinate and satisfy following genericity conditions: 
\begin{itemize}
\item Each pair of curves cross at each point of intersection.
\item No three curves share a common point of intersection. 
\item There are finitely many crossings.
\end{itemize}


\paragraph{Combinatorial type.}\label{ctdef}
Let $\sym_m$ be the symmetric group on $m$ elements and $[m] = \{1,\dots,m\}$. 
Given $i \in [m-1]$, the \df{adjacent transposition}  $\tau_i \in \sym_m$ is the permutation interchanging the $i$'th and $i{+}1$'st entries, 
\[ \tau_i(x_1,\dots,x_m) \ = \ (x_1,\dots,x_{i-1},x_{i+1},x_{i},x_{i+2},\dots,x_m). \]

Let $H(\tau_i) = i$ denote the height of an adjacent transposition. 
A \df{swap sequence} $\sigma : [N] \to \sym_m$ is any sequence of adjacent transpositions such that $\sigma_N\circ \dots \circ\sigma_1$ is the identity permutation.

Fix an index set $\indx$ of size $n$.  A \df{swap pair} $(\rho,\sigma)$ on $\indx$ is a bijection $\rho : [n] \to \indx$ together with a swap sequence $\sigma : [N] \to \sym_n$.
We define an equivalence relation $(\swapeq)$ on swap pairs as follows. 
Let $(\rho',\sigma') \swapeq (\rho,\sigma)$ when $(\rho',\sigma')$ can be obtained from $(\rho,\sigma)$ by performing any sequence of the following two \df{elementary operations} \newline
\vbox{
\begin{itemize}
\item a \df{cyclic shift}
\[ \rho' = \tau_{\sigma_1}(\rho), \quad \sigma\mathrlap{'}_i = \sigma_{(i{+}1 \ \text{mod} \ N)} \]
\item an \df{independent transposition}
\[\rho' = \rho, \quad \sigma' = \tau_i(\sigma) \quad \text{provided} \quad |H(\sigma_i)-H(\sigma_{i{+}1})|>1. \] 
\end{itemize}
}
A \df{combinatorial type} $\coty$ on $\indx$ is the equivalence class 
$\coty = \{(\rho',\sigma'): (\rho',\sigma') \swapeq (\rho,\sigma)\}$ 
of a swap pair $(\rho,\sigma)$.

To define the combinatorial type of a system $\mc S$, 
we order the crossings of $\mc S$ lexicographically in $\mb S^1 \times \mb R^1$ where $\mb S^1$ is ordered according to the standard parametrization by the half-open interval $(0,2\pi]$. 
Let $\rho$ be the order of the indices of each curve from bottom to top before the first crossing of the system.
Let $\sigma$ be the swap sequence corresponding to each crossing. 
That is, let $H(\sigma_i)$ be 1 plus the number of curves below the $i$'th crossing of $\mc S$.  
Observe that the sequence $\sigma_i\circ \dots \sigma_1(\rho)$ for $i=0,1,\dots,N$ records the order of the curves in a sweep of the cylinder.
The \df{combinatorial type} $\ct(\mc S)$ of a system $\mc S$ is the equivalence class of its swap pair $(\rho,\sigma)$. 
The combinatorial type of an arrangement $\mc A$ is that of its dual system, and by slight abuse of notation, we write $\ct(\mc A) = \ct(\mc A^*)$.

\begin{figure}[h]
\centering

\begin{tikzpicture}

\begin{scope}[yscale=.5,xscale=.7]

\draw[thick,cyan]
(0,4) node[left] {\ft $d$}
\g{3} \ds \du \ue
;

\draw[thick,blue]
(0,3) node[left] {\ft $c$}
\g{1} \ds \du \ug{1} \ud \de
;

\draw[thick,violet]
(0,2) node[left] {\ft $b$}
\ds \de \g{2} \us \ue
;

\draw[thick,black]
(0,1) node[left] {\ft $a$}
\us \ug{1} \ud \de \ds \de
;

\node at (0,0) {\ft $0$};
\node at (6,0) {\ft $2\pi$};

\end{scope}

\node 
at (9,1.25) 
{
$\begin{array}{r@{\ }c@{\ }l}
\rho &=& (a,b,c,d) \\
H(\sigma) &=& (1,2,2,3,1,3) 
\\
\overline{\rho\sigma} &=& ((b,a),(c,a),(a,c),(d,c),(a,b),(c,d))
\end{array}$
};

\end{tikzpicture}
\caption{\ft
A system with its swap pair $(\rho,\sigma)$ and its incidence sequence $\overline{\rho\sigma}$.\newline
Note that systems are drawn as viewed from outside the cylinder, so counter-clockwise is to the right.
}\label{eg_comb_type}
\end{figure}
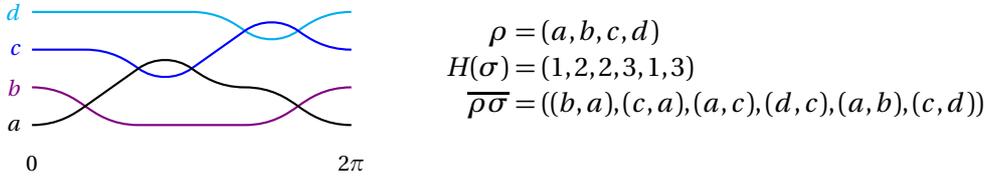

The \df{incidence sequence} $\overline{\rho\sigma} : [N] \to \indx^2$ of a swap pair $(\rho,\sigma)$ records the ordered pair of indices transposed by the action of each swap, 
\[ \overline{\rho\sigma}{}_i = (x_{H(\sigma_i)+1},x_{H(\sigma_i)}) \quad \text{where} \quad 
x = \sigma_{i-1}\circ \cdots \circ \sigma_1(\rho).\]
Note that the incidence sequence of equivalent swap pairs have the same multi-set of entries. 

The \df{layers} of a system are the connected components of the union of curves of the system. 
Analogously, the layers of a combinatorial type are the connected components of the graph on $\indx$ defined by its incidence sequence. 
The \df{depth} of a combinatorial type is the number of layers excluding isolated vertices, and the depth 1 case is called \df{non-layered}.  

\begin{remark}
Orientable combinatorial types are always non-layered.
\end{remark}

\paragraph{Realization spaces.} \label{realize this}

The \df{full realization space} $\fullrs(\coty)$ of a combinatorial type $\coty$ is defined by
\[ \fullrs(\coty) := \{\mc A\in\mc U_\indx: \ct(\mc A) = \coty\} \]
where $\mc U_\indx$ is the set of all arrangements of bodies indexed by $\indx$.  
The Hausdorff metric $d_H$ on compact subsets of 
$\mb{R}^2$ induces a metric on $\fullrs(\coty)$ by taking the maximum distance between bodies having the same index. 
That is, for ${\mc A} = \{A_{ i}\}_{ i\in \indx}$ and  ${\mc B} = \{B_{ i}\}_{ i\in \indx}$, 
\[d({\mc A}, {\mc B}) = \max_{  i \in \indx} d_H(A_{ i}, B_{ i}) \] 

\begin{remark}
The map that takes a convex body to its support function is an isometry from the space of convex bodies with the Hausdorff metric to the space of support functions on $\mb{S}^1$ with the supremum metric. 
\end{remark}

Depending on context, it may be convenient to regard realizations of a fixed combinatorial type as ``the same'' when they are related by a projective transformation. 
Let $\mc A \projeq \mc B$ when they are related by an admissible projectivity; that is, an invertible projective transformation $\pi$ such that $\pi(A_{ i}) = B_{ i}$ for all $ i \in \indx$ and $\pi$ is bounded and preserves orientation on the convex hull of ~$\bigcup \mc A$.  The \df{(projective) realization space}, which we may simply call the ``realization space'', is the quotient space 
\[\projrs(\coty) := \fullrs(\coty)/\projeq .\]

By a $k$-gon we mean a convex polygon with {\em at most} $k$ vertices.  The \df{full $k$-gon realization space} is given by
\[ \fullrs_k(\coty) := \{\mc A\in \fullrs(\coty): \forall  i\in \indx.\ A_{ i} \text{ is a $k$-gon}\} \] 
Similarly, we have the \df{(projective) $k$-gon realization space} $\projrs_k(\coty):= \fullrs_k(\coty)/\projeq$. 

\begin{remark}\label{circprojrs}
By considering the direction of some line $\ell$ passing through an arrangement $\mc A$, the admissible projectivities of $\mc A$ factor into a contractible set of projective transformations fixing the direction of $\ell$ and rotations of the plane $SO(2)$.  As such, $\fullrs(\kappa)$ is homotopic to $\projrs(\kappa) \times \mb{S}^1$.
\end{remark}

We now state the main theorems on realization space of arrangements of convex bodies. 

\begin{theorem}\label{realizations}
If ~$\Omega$ is a non-layered combinatorial type (in particular, if ~$\Omega$ corresponds to a simple order types),
then its realization space~$\projrs(\coty)$ is contractible. 
Moreover, if ~$\coty$ has depth $d > 1$, then ~$\projrs(\coty)$ is homotopic to a product of $d{-}1$ circles. 
\end{theorem}

\begin{theorem}\label{kgonrs}
For every primary semialgebraic set ~$\mc Z$ and every positive integer $k$, there exists a combinatorial type ~$\coty$ such that its $k$-gon realization space~$\projrs_k(\coty)$ is homotopic to ~$\mc Z$. 
\end{theorem}

\section{Realizing order types} \label{counting}

\paragraph{Orientability and order types.}

We start by showing a bijective correspondence between orientable arrangements and simple order types.  

\begin{proof}[Proof of Theorem \ref{orientable}]
the Folkman-Lawrence Representation Theorem gives a bijective correspondence between simple order types $\chi$ and equivalence classes of uniform acyclic pseudocircle arrangements on the sphere $\mb{S}^2$.  
We extend this bijection to equivalence classes of orientable arrangements by mapping the dual support systems to pseudocircle arrangements.
Note that a positive orientation on the sphere is chosen; if we were to forget this orientation, then each equivalence class of pseudocircle arrangements would correspond to the pair $\{\chi,-\chi\}$.

Let $\mc A = \{A_1,\dots,A_n\}$ be an orientable arrangement.  Let $\phi : \mb{S}^1 \times \mb{R}^1 \to \mb{S}^2$ be the compactification of the cylinder defined by adding one point $p_{+\infty}$ at $+\infty$ and another point $p_{-\infty}$ at $-\infty$ where ${(\theta,y) \to p_{\pm\infty} \in \mb{S}^2}$ as ${y \to \pm\infty}$. 
The image of the dual support system $\phi(\mc A^*) = \{\phi(A_1^*),\dots,\phi(A_n^*)\}$ is now a uniform acyclic pseudocircle arrangement; $\phi(\mc A^*)$ is a uniform pseudocircle arrangement by the genericity assumptions and $\phi(\mc A^*)$ is acyclic since the region of cylinder that is above every curve of $\mc A^*$ is now to the left of every curve of $\phi(\mc A^*)$.

For the other direction, let $\mc S = \{S_1,\dots,S_n\}$ be a uniform acyclic pseudocircle arrangement, and let $p_{+\infty}$ be a point to the left of each curve and $p_{-\infty}$ be a point to the right of each curve. 
Pseudocircle arrangements can be swept, meaning a path $\gamma_t$ from the point $p_{-\infty}$ to the point $p_{+\infty}$ can be continuously deformed while keeping its end-points fixed such that it  passes through all other points on the sphere exactly once returning to its original position and it always intersects each pseudocircle at exactly one point \cite[Theorem~2.9]{GPsemi}. 
This defines a homeomorphism from ${\mb{S}^2 \setminus \{ p_{+\infty}, p_{-\infty}\}}$ to the cylinder such that the image of each pseudocircle is the graph of a function $f_i : \mb{S}^1 \to \mb{R}^1$.  Each of the $f_i$ can then be approximated by a smooth function $h_i$ while preserving their order above each point on the circle, 
$\forall \theta \in \mb{S}^1.\ f_{i_1}(\theta) \leq \dots \leq f_{i_n}(\theta)\ \Leftrightarrow \ h_{i_1}(\theta) \leq \dots \leq h_{i_n}(\theta)$.
By Remark~\ref{blashke}, we may now assume the $h_i$ are support functions of an arrangement $\mc A$.  With this, $\phi(\mc A^*)$ is a pseudocircle arrangement of the same order type as $\mc S$.

An orientable arrangement $\mc A$ now has a simple order type given by $\phi(A^*)$, and the orientations of a triple $A_i,A_j,A_k$ are given by the order $\phi(A_i^*),\phi(A_j^*),\phi(A_k^*)$ appear on the boundary of the region to the left of each curve.  
By Theorem~\ref{combequivtype} two orientable arrangements have the same combinatorial type if and only if they have the same order type.  And, every simple order type can be realized by an orientable arrangement. 
\end{proof}

\paragraph{Counting arrangements of polygons.} 
Here we give an upper bound on the number of combinatorial types that can be realized by $k$-gons, and bounds on the value of $k_n$ needed to realize all simple order types on $[n]$ by $k_n$-gons.
Our results are based on the following. 

\begin{theorem}[Goodman and Pollack \cite{gp-upper}] \label{ordertype:count}
Let $t(n)$ denote the number of distinct order types (not necessarily simple) on $[n]$ that can be realized by points. For some constants $c_1$, $c_2$, \[ 2^{4n\log n +c_1 n}  \leq t(n)  \leq 2^{4n \log n + c_2 n} . \]
\end{theorem}

\begin{theorem}[Felsner and Valtr \cite{fels-valt}] \label{abs:count}
Let $t_\text{ot}(n)$ denote the number of distinct simple order types on $[n]$. \[ 2^{0.188 n^2} \leq t_\text{ot}(n) \leq 2^{0.6571 n^2} . \]
\end{theorem}

\begin{theorem} \label{k-real}
Let $t_k(n)$ denote the number of distinct combinatorial types on $[n]$ that can be realized by $k$-gons. For some constant $c$, \[ t_k(n) \leq 2^{c kn (\log n + \log k)} . \]
\end{theorem}

\begin{proof}
Since the combinatorial type of an arrangement of $n$ $k$-gons is determined by the order type of the $kn$ vertices of the arrangement, we have the inequality, $t_k(n) \leq t(kn)$. Specifically, for every combinatorial type $\coty$ that can be realized by $k$-gons, fix a realization by $k$-gons $\mc A = \{A^1,\dots,A^n\}$.  Then, for each $k$-gon $A^s$, fix a labeling of vertices so that $A^s$ is the convex hull of a point set $\{p_1^s,\dots,p_k^s\}$.  Now associate the order type of the points $\{p_1^1,\dots,p_k^n\}$ to $\coty$.  In this way, we define an injective map from the set of combinatorial types on $[n]$ that can be realized by $k$-gons to the set of order types on $\{({}_1^1),\dots,({}_k^n)\}$ that can be realized by points.
The upper bound on the number of combinatorial types now follows from Theorem~\ref{ordertype:count}.
\end{proof} 

\begin{theorem} \label{N-gons}
A combinatorial type ~$\coty$ with $N$ common supporting tangents 
can be realized by $N$-gons.    
\end{theorem}

\begin{proof}
Recall that the wiring digram of a sequence of adjacent transpositions $\sigma$ is a family of polygonal paths that cross according to $\sigma$ \cite{goodburr}. 
Informally, we will define an arrangement of $N$-gons by ``wrapping'' a wiring diagram of $\sigma$ around a large circle. 

Let $(\rho,\sigma) \in \Omega$ and $\pi_t = \sigma_t \circ \dots \circ \sigma_1$. 
Choose $r$ sufficiently large 
and let $p_{i,t} = (2\pi t/N; r +\pi\mathrlap{_t}^{-1}(i))$ in polar coordinates, so that each point set $P_i = \{p_{i,1},\dots,p_{i,N}\}$ for $i \in [n]$ is in convex position.  For this, $r \geq {(1-\cos(2\pi/N))^{-1}}$ would suffice. 
Let $A_i = \conv(P_i)$. 
Since the support functions $h_i$ of the $A_i$ now satisfy 
\[ h_{\rho\pi\mathrlap{_t}^{-1}(1)}(\theta_t) < \dots < h_{\rho\pi\mathrlap{_t}^{-1}(n)}(\theta_t) \]
for $\theta_t = 2\pi t/N$ with only once crossing between $\theta_t$ and $\theta_{t+1}$, 
the swap pair of the arrangement $\mc A = \{A_1,\dots,A_n\}$ is $(\rho,\sigma)$. 
Hence, $\mc A$ is an arrangement of $N$-gons with combinatorial type $\coty$.
\end{proof}

\begin{proof}[Proof of Theorem \ref{k-bounds}.]
The upper bound follows immediately from Theorem~\ref{N-gons} and the observation than an orientable arrangement on $[n]$ always has $N = n(n-1)$ common supporting tangents. 
For the lower bound, observe from Theorems~\ref{abs:count} and~\ref{k-real} that for any fixed $k$, the number of simple order types grows faster than the number of combinatorial types that can be realized by $k$-gons.  Specifically, there is some constant $c_1$ such that if $k \leq c_1\frac{n}{\log n}$ then $t_k(n)  < t_\text{ot}(n)$, which implies that some simple order type on $[n]$ cannot be realized by $k$-gons.
\end{proof}

\section{Contractibility} \label{contract}

To show contractibility, we construct a standard arrangement of convex bodies for each combinatorial type by defining its dual support system. 
We then show that the full realization space $\fullrs(\coty)$ is equivariantly homotopic to a circle $\mb{S}^1$ by defining a deformation retraction to the subspace of rotated copies of the standard arrangement. 
By \df{equivariantly} homotopic we mean that the corresponding homotopy maps commute with $SO(2)$.  
We then pass to the (projective) realization space $\projrs(\coty)$ by identifying arrangements related by admissible projectivities. 
Since rotations are admissible projectivities, this defines a deformation retraction from $\projrs(\coty)$ to a point.

The deformation retraction from $\fullrs(\coty)$ to a circle will proceed in two steps; see \fig~\ref{compressed}.  First in Lemma~\ref{lemcontract1}, we deform the support system of a given arrangement to a system of the same combinatorial type that depends only on the (angular) position of each crossing.  We can then consider just the positions of the crossings and ignore the rest of the geometry of the system.  Second in Lemma~\ref{lemcontract2}, we move the crossings to a set of standard positions that depend only on: the given combinatorial type and the position of a certain crossing that we fix.  The set of possible standard systems we get in the end is parametrized by the position of this fixed crossing, which defines an embedding of the circle in $\fullrs(\coty)$. 
The first deformation retraction depends on the following remark. 

\begin{remark}\label{mink}
For any pair of convex bodies $A$ and $B$, $(A+B)^*=A^*+B^*$ with Minkowski addition on the left and addition of the support functions defining the curves on the right. And, for $t \geq 0$, $(t A)^* = t (A^*)$. 
Hence, the set of all support functions is a convex cone.  That is, if $h_1$ and $h_2$ are support functions, then so is $t_1h_1 +t_2h_2$ for $t_i \geq 0$.
Note however, that the set of dual support systems of a fixed combinatorial type is not a convex set. 
\end{remark}

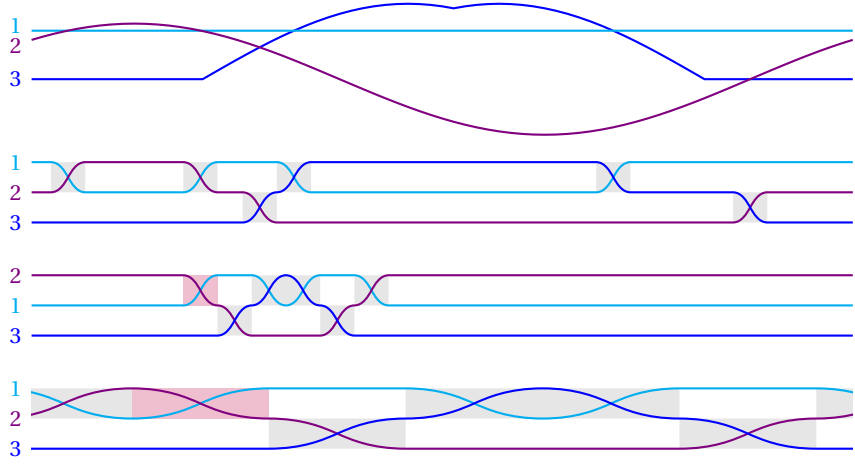
\begin{figure}[h] \centering
\begin{tikzpicture}

\begin{scope}[shift={(-5.5cm,-.2cm)},xscale=.03,yscale=1]

%

\draw[thick, blue]
(0,0) node[left] {\ft 3} -- (75,0)
\foreach \t in {80,85,...,185}
{ -- (\t,{cos(\t-165)}) }
\foreach \t in {190,195,...,295}
{ -- (\t,{cos(\t-205)}) }
 -- (360,0)
;

\draw[thick, cyan]
(0,{cos(50)}) node[left,yshift=2pt] {\ft 1} -- (360,{cos(50)})
;

\draw[thick, violet]
(0,{(sin(25)/sin(35))*cos(-45)}) node[left,yshift=-2pt] {\ft 2}
\foreach \t in {5,10,...,360}
{ -- (\t,{(sin(25)/sin(35))*cos(\t-45)}) }
;

\end{scope}

\begin{scope}[shift={(-5.5cm,-2.5cm)},xscale=.225,yscale=.4]

\fill[black!10]
(8.5/7.5,2) rectangle (23.5/7.5,3)
(66.5/7.5,2) rectangle (81.5/7.5,3)
(92.5/7.5,1) rectangle (107.5/7.5,2)
(107.5/7.5,2) rectangle (122.5/7.5,3)
(247.5/7.5,2) rectangle (262.5/7.5,3)
(307.5/7.5,1) rectangle (322.5/7.5,2)
;

\draw[thick,cyan]
(0,3) node[left] {\ft 1} \g{8.5/7.5} \ds \de \g{43/7.5} \us \ue \g{26/7.5} \ds \de \g{125/7.5} \us \ue \g{97.5/7.5}
;

\draw[thick,violet]
(0,2) node[left] {\ft 2} \g{8.5/7.5} \us \ue \g{43/7.5} \ds \de \g{11/7.5} \ds \de \g{200/7.5} \us \ue \g{37.5/7.5}
;

\draw[thick,blue]
(0,1) node[left] {\ft 3} \g{92.5/7.5} \us \ue \us \ue \g{125/7.5} \ds \de \g{45/7.5} \ds \de \g{37.5/7.5}
;

\end{scope}

\begin{scope}[shift={(-5.5cm,-4cm)},xscale=.225,yscale=.4]

\fill[purple!25]
(66.5/7.5,2) rectangle (81.5/7.5,3)
;

\fill[black!10]
(81.5/7.5,1) rectangle (96.5/7.5,2)
(96.5/7.5,2) rectangle (126.5/7.5,3)
(126.5/7.5,1) rectangle (141.5/7.5,2)
(141.5/7.5,2) rectangle (156.5/7.5,3)
;

\draw[thick,cyan]
(0,2) node[left] {\ft 1} \g{66.5/7.5} \us \ue \g{15/7.5} \ds \de \g{0} \us \ue \g{15/7.5} \ds \de \g{203.5/7.5}
;

\draw[thick,violet]
(0,3) node[left] {\ft 2} \g{66.5/7.5} \ds \de \g{0} \ds \de \g{30/7.5} \us \ue \g{0} \us \ue \g{203.5/7.5}
;

\draw[thick,blue]
(0,1) node[left] {\ft 3} \g{81.5/7.5} \us \ue \us \ue \g{0} \ds \de \g{0} \ds \de \g{218.5/7.5}
;

\end{scope}

\begin{scope}[shift={(-5.5cm,-5.5cm)},xscale=.225,yscale=.4]
\begin{scope}[xscale=4]
\clip
(0,.5) rectangle (90/7.5,3.5)
;

\fill[purple!25]
(11/7.5,2) rectangle (26/7.5,3)
;

\fill[black!10]
(26/7.5,1) rectangle (41/7.5,2)
(41/7.5,2) rectangle (71/7.5,3)
(71/7.5,1) rectangle (86/7.5,2)
(86/7.5,2) rectangle (101/7.5,3)
(-4/7.5,2) rectangle (11/7.5,3)
;

\draw[thick,cyan]
(-4/7.5,3) 
\ds \de \us \ue \g{2} \ds \de \us \ue \g{2} \ds \de
;

\draw[thick,violet]
(-4/7.5,2) 
\us \ue \ds \de \ds \de \g{4} \us \ue \us \ue
;

\draw[thick,blue]
(-4/7.5,1) 
\g{4} \us \ue \us \ue \ds \de \ds \de \g{2}
;

\end{scope}

\draw[thick,cyan]
(0,3) node[left] {\ft 1};
\draw[thick,violet]
(0,2) node[left] {\ft 2};
\draw[thick,blue]
(0,1) node[left] {\ft 3};

\end{scope}

\end{tikzpicture}
\caption{\ft
From top to bottom: 
(1) The dual support system of an arrangement $\mc A$ of combinatorial type $\coty$; compare with \fig \ref{eg_dual}. 
(2) The system $\alpha^* \projrsvc(\mc A)$ of Lemma~\ref{lemcontract1}, which depends only on the angular positions of the crossings of $\mc A^*$. 
(3) The system $\alpha^* W(\coty,\theta,\delta)$ obtained by fixing the marked crossing of $\alpha^* \projrsvc(\mc A)$, and rotating all other crossings clockwise. 
(4) The standard system $\alpha^* W(\coty,\theta)$ of Lemma~\ref{lemcontract2}, which depends only on the angular position $\theta$ of the fixed crossing and $\coty$. 
}
\end{figure}

\paragraph{Support configurations.}

The \df{support configuration} of an
arrangement $\mc A$ indexed by $\indx$ is a labeled vector configuration ${\projrsvc({\mc A}) \subset \indx^2 \times \mb{S}^1}$ which contains a triple $( i, j,\theta)$ if bodies $A_{ i},A_{ j}$ have a common supporting tangent line $\ell$ that first meets $A_{ i}$ and then meets $A_{ j}$ and has outward normal vector $\theta$. 
We say labels $(i,j)$, $(i',j')$ are \df{disjoint} when $\{i,j\} \cap \{i',j'\} = \emptyset$.
Note that a unit vector $\theta$ may appear in multiple elements of $\projrsvc({\mc A})$ with disjoint labels. 
Dually, $\projrsvc({\mc A})$ corresponds to the crossings of $\mc A^*$. 
Specifically, $( i, j,\theta) \in \projrsvc({\mc A})$ when curves $A_{ i}^*$ and $A_{ j}^*$ cross at $\theta$ with $A_{ i}^*$ crossing downward and $A_{ j}^*$ crossing upward. 
That is, the respective support functions $f_{ i}, f_{ j}$ of $A_{ i},A_{ j}$ are equal at $\theta$ and $f_{ j}-f_{ i}$ is increasing at $\theta$. 

Observe that the support configuration of an arrangement determines the combinatorial type of that arrangement.  
For a given combinatorial type $\coty$, we will define its \df{support configuration space} $\vcrs(\coty)$, which will turn out to be the set of support configurations of all arrangements realizing $\coty$.  
We first define the set of labeled configurations $\vcrs(\rho,\sigma)$ corresponding to a given swap pair $(\rho,\sigma)$.  
Recall that $\overline{\rho\sigma}$ records the ordered pairs of indices transposed by $\sigma$ acting sequentially on $\rho$.  Observe that if $(\rho,\sigma)$ is the swap pair of a system, then $\overline{\rho\sigma}{}_i$ for $i \in [N]$ is the labeling of the $i$-th crossing of the system.  Recall also that we order $\mb{S}^1$ by the parametrization by $(0,2\pi]$.  Let 
\[ 
\vcrs(\rho,\sigma) \coloneqq \left\{ \left\{ (\overline{\rho\sigma}{}_i,\,\theta_i) :\ {i \in [N]} \right\} :\
\begin{array}{l@{\vspace{2pt}}}
\theta_i \in \mb S^1,\ \theta_i \leq \theta_{i+1},\\ \theta_i = \theta_{i+1} \Rightarrow |H(\sigma_i)-H(\sigma_{i+1})| > 1
\end{array} \right\},
\]
\[
\vcrs(\coty) \coloneqq \bigcup_{(\rho,\sigma)\in \coty} \vcrs(\rho,\sigma). 
\]
Note that a vector $\theta \in \mb S^1$ might appear multiple times in $\vcrs(\coty)$ with different labels provided the pairs of indices in the labels are disjoint.  

We define a metric on $\vcrs(\coty)$ as follows. 
For a given support configuration $X$ and a given ordered pair of indices $(i,j)\in {\indx \choose 2}$, let 
$X_{(i,j)}:=\{\theta \in \mb S^1 : (i,j,\theta)\in X\}$. 
For two support configurations, $X,Y \subset \indx^2 \times \mb{S}^1$, 
\[d(X,Y)= \max_{(i,j)\in {\indx \choose 2}}d_H( X_{(i,j)}, Y_{(i,j)})\] 
where the distance between two direction vectors is given by their angle and $d_H$ is the corresponding Hausdorff metric on sets. 

\begin{lemma}\label{lemcontract1} 
For any combinatorial type $\coty$, the full realization space $\fullrs(\coty)$ is non-empty and equivariantly homotopic to the support configuration space $\vcrs(\coty)$.  
\end{lemma}

\begin{proof}
For $\mc{A} \in \fullrs(\coty)$ with swap pair $(\rho,\sigma)$, 
we have $\projrsvc(\mc A) \in \vcrs(\rho,\sigma) \subset \vcrs(\coty)$, so assigning each arrangement to its support configuration defines a map $\projrsvc \colon \fullrs(\coty) \to \vcrs(\coty)$, which will be one direction of the homotopy equivalence.  

For the other direction, 
we define an embedding ${\alpha : \vcrs(\coty) \to \fullrs(\coty)}$.  For each labeled configuration $V \in \vcrs(\coty)$, we construct a system of curves $\alpha^*(V) = {\{ A^*_i : i \in \indx \}}$ where $A^*_i = f_i(\mb{S}^1)$, $f_i : \mb{S}^1 \to \mb{R}^1$, and show that $\alpha^*(V)$ is the dual support system of an arrangement $\alpha(V)$ that has support configuration $V$. 
The system $\alpha^*(V)$ that we construct may be regarded as a smooth analog of Goodman's wiring diagram \cite{goodburr}. 

Fix $V \in \vcrs(\coty)$, let $V_{ i}\subset \mb{S}^1$ denote the vectors of $V$ with labels involving $ i$, and 
let $\delta$ be the minimum angular distance between any two vectors of $V$ with non-disjoint labels. 
For $v = (i,j,\theta) \in V$ define the open arc $ \Theta(v) := \left(\theta - \nicefrac{\delta}{2},\; \theta + \nicefrac{\delta}{2} \right) \subset \mb{S}^1$. Now define $f_{ i}$ to be constant on the complement of the arcs $\Theta(V_{ i})$, and
to smoothly increase or decrease by $\pm 1$ symmetrically about $\theta$ in each arc $\Theta(v)$ according to the label on $v \in V_{ i}$; that is, $f_{ i}$ increases on $\Theta(v)$ if $( j, i,\theta) \in V$ and decreases if $( i, j,\theta) \in V$ for some $ j$.\footnote{The definition of $f_{ i}$ on $\Theta(v)$ is irrelevant as long as $f_{ i}$ is $C^2$-smooth, monotonic, symmetric about $\theta$, and varies continuously with respect to $V$.  A cubic spline would suffice for this.} 
We claim that each $f_i$ is well defined up to an additive constant, and these constants can be chosen so that each pair $f_{ i},f_{ j}$ coincides on $V_{ i} \cap V_{ j}$ and nowhere else. 

Since $\sigma_N \cdots \sigma_1$ is the identity permutation, each $f_{ i}$ increases the same number of times as it decreases on arcs in $\Theta(V_{ i})$, so traversing once around $\mb{S}^1$ results in no net change in the value of $f_{ i}$, which is therefore well defined up to an additive constant. 

Let $c$ be some constant to be determined later.  
For $\xi \in \mb{S}^1$ ordered by $(0,2\pi]$, let $j(\xi)$ be the largest integer for which $\theta_{j(\xi)} < \xi$ where $\theta_j$ is the angle of the $j$'th vector of $V$. 
We will show that letting $f_i(\xi) = c + \sigma_{j(\xi)} \dots \sigma_1 \rho(i)$ on the complement of $\Theta(V_i)$ satisfies our claim.  
For consecutive angles $\theta_a, \theta_b \in V_i$,  
the labels on the vectors of $V$ in the interval $[\theta_a +\nicefrac{\delta}{2},\theta_b -\nicefrac{\delta}{2}]$ do not involving $i$, so $i$ is remains fixed by the corresponding adjacent transpositions of the swap sequence $\sigma_{j(\theta_a +\nicefrac{\delta}{2})+1},\dots,\sigma_{j(\theta_b -\nicefrac{\delta}{2})}$, which implies $f_i$ is constant on this interval.  Since the $f_i$ are defined to increase or decrease on $\Theta(V_i)$ according to the action of the swap sequence, this completely determines a smooth function $f_i$.  Since the $f_i$ increase or decrease symmetrically in each arc of $\Theta(V_i)$, each pair $f_i$, $f_j$ coincide exactly on $V_{ i} \cap V_{ j}$.

To fix the constant $c$, let 
\[ \min_{ ( i,\theta) \in \indx \times \mb{S}^1} (f_{ i}(\theta) + f\mathrlap{{}_{ i}}''(\theta)) = 1. \] 
By Remark~\ref{blashke}, the system $\alpha^*(V)$ defined by the functions $f_i$ is the dual support system of an arrangement $\alpha(V) \in \fullrs(\coty)$ that is uniquely and continuously determined by $V \in \vcrs(\coty)$, and $\projrsvc(\alpha(V)) = V$. 
This gives us a subspace ${\alpha}(\coty) := \{ \alpha(V) : V \in \vcrs(\coty)\} \subset \fullrs(\coty)$ that is homeomorphic to $\vcrs(\coty) = \projrsvc(\fullrs(\coty))$. 
For $\mc A \in \fullrs(\coty)$ define 
${\mc A}_t := t \alpha({\projrsvc(\mc A)}) + (1-t){\mc A}$ for $0\leq t \leq 1$ by Minkowski addition on each body of the arrangement. Since $\projrsvc(\mc A) = \projrsvc (\alpha({\projrsvc(\mc A)}))$ and, as we linearly interpolate between two systems with the same crossings, the crossings remain fixed by Remark~\ref{mink}, $\projrsvc(\mc A _t)$ is constant for all $t \in[0,1]$. Thus, $\alpha(\coty)$ is an equivariant deformation retract of $\fullrs(\coty)$. 
\end{proof}

\paragraph{Local sequences and associated tableau.}
We now introduce an encoding of the combinatorial type extending the local sequences of a point set. 
This encoding will be used in several proofs, including in later sections of the paper. 

For a system $\mc S$, the \df{local sequence} $\lambda_{ i} = (\lambda_{i,1},\dots,\lambda_{i,n_i})$ of $i \in \indx$ lists the indices of the curves that $S_{ i}$ crosses ordered by $(0,2\pi]$; see \fig~\ref{eg_bump}.  
Analogously for a pair $(\rho,\sigma) \in \coty$, the local sequence $\lambda_{ i}$ lists the indices $\lambda_{i,j}$ appearing together with $ i$ as part of a pair $(\lambda_{i,j},i)$ or $(i,\lambda_{i,j})$ in the incidence sequence $\overline{\rho\sigma}$.  The \df{associated tableau} $\Lambda$ of $(\rho,\sigma)$ is the tableau whose rows are the local sequences for $(\rho,\sigma)$ ordered by $\rho$ from bottom to top. We say $\Lambda$ is a \df{tableau representation} of the combinatorial type $\coty$. The local sequences and associated tableau of a system or arrangement are that of its swap pair.

\begin{figure}[h]
\centering

\begin{tikzpicture}

\begin{scope}[yscale=.5,xscale=.7]

\draw[thick,cyan]
(0,4) node[left] {\ft $d$}
\g{3} \ds \du \ue
;

\draw[thick,blue]
(0,3) node[left] {\ft $c$}
\g{1} \ds \du \ug{1} \ud \de
;

\draw[thick,violet]
(0,2) node[left] {\ft $b$}
\ds \de \g{2} \us \ue
;

\draw[thick,black]
(0,1) node[left] {\ft $a$}
\us \ug{1} \ud \de \ds \de
;

\node at (0,0) {\ft $0$};
\node at (6,0) {\ft $2\pi$};

\end{scope}

\begin{scope}[shift={(0,-2)},scale = .4,every node/.style={anchor=mid}]
\draw 
(0,0) -- (4,0)
(0,1) -- (4,1)
(0,2) -- (4,2)
(0,3) -- (4,3)
(0,4) -- (2,4)
(0,0) -- (0,4)
(1,0) -- (1,4)
(2,0) -- (2,4)
(3,0) -- (3,1)  (3,2) -- (3,3)
(4,0) -- (4,1)  (4,2) -- (4,3)
;
      \node at (.5,.5) {\ft $b$};
      \node at (1.5,.5) {\ft $c$};
      \node at (2.5,.5) {\ft $c$};
      \node at (3.5,.5) {\ft $b$};
      \node at (.5,1.5) {\ft $a$};
      \node at (1.5,1.5) {\ft $a$};
      \node at (.5,2.5) {\ft $a$};
      \node at (1.5,2.5) {\ft $a$};
      \node at (2.5,2.5) {\ft $d$};
      \node at (3.5,2.5) {\ft $d$};
      \node at (.5,3.5) {\ft $c$};
      \node at (1.5,3.5) {\ft $c$};
\path
(-.5,3.5) node[cyan] {\ft $d$}
(-.5,2.5) node[blue] {\ft $c$}
(-.5,1.5) node[violet] {\ft $b$}
(-.5,.5) node[black] {\ft $a$}
;

\end{scope}

\end{tikzpicture}
\hspace{1.5cm}
\begin{tikzpicture}

\begin{scope}[yscale=.5,xscale=.7]

\draw[thick,cyan]
(0,4) node[left] {\ft $d$}
\g{2} \ds \du \ue \g{1}
;

\draw[thick,blue]
(0,3) node[left] {\ft $c$}
\ds \du \ug{1} \ud \de \g{1}
;

\draw[thick,violet]
(0,1) node[left] {\ft $b$}
\g{3} \us \ud \de
;

\draw[thick,black]
(0,2) node[left] {\ft $a$}
\us \ud \de \ds \du \ue
;

\node at (0,0) {\ft $0$};
\node at (6,0) {\ft $2\pi$};

\end{scope}

\begin{scope}[shift={(0,-2)},scale = .4,every node/.style={anchor=mid}]
\draw 
(0,0) -- (2,0)
(0,1) -- (4,1)
(0,2) -- (4,2)
(0,3) -- (4,3)
(0,4) -- (2,4)
(0,0) -- (0,4)
(1,0) -- (1,4)
(2,0) -- (2,4)
(3,1) -- (3,3)
(4,1) -- (4,3)
;
      \node at (.5,.5) {\ft $a$};
      \node at (1.5,.5) {\ft $a$};
      \node at (.5,1.5) {\ft $c$};
      \node at (1.5,1.5) {\ft $c$};
      \node at (2.5,1.5) {\ft $b$};
      \node at (3.5,1.5) {\ft $b$};
      \node at (.5,2.5) {\ft $a$};
      \node at (1.5,2.5) {\ft $a$};
      \node at (2.5,2.5) {\ft $d$};
      \node at (3.5,2.5) {\ft $d$};
      \node at (.5,3.5) {\ft $c$};
      \node at (1.5,3.5) {\ft $c$};
\path
(-.5,3.5) node[cyan] {\ft $d$}
(-.5,2.5) node[blue] {\ft $c$}
(-.5,.5) node[violet] {\ft $b$}
(-.5,1.5) node[black] {\ft $a$}
;

\end{scope}

\end{tikzpicture}

\caption{\ft 
\textbf{Left:} A system and its associated tableau of local sequences $\Lambda$. \textbf{Right:} The tableau $\bump(\Lambda,\{a,b\})$ and a corresponding system.
}\label{eg_bump}
\end{figure}
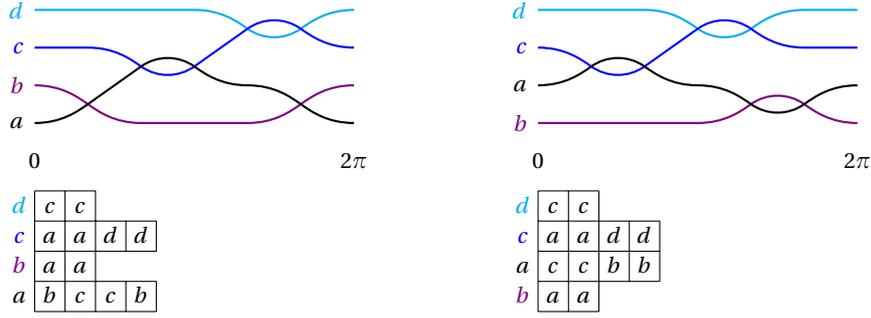

Like in the definition of combinatorial type, we define an equivalence relation on associated tableau. We say a pair $ j, k \in \indx$ are \df{adjacent} in a tableau $\Lambda$ with rows $\lambda_{ i}$ when \[ \lambda_{ j} = ( k,\lambda_{ j,2},\dots,\lambda_{ j,n_j}) \quad \text{and} \quad \lambda_{ k} = ( j,\lambda_{ k,2},\dots,\lambda_{ k,n_k}).\] In this case, \df{bumping} $\{ j, k\}$ sends $\Lambda$ to the tableau $\Lambda' = \bump(\Lambda,\{ j, k\})$ with rows 
\[ \lambda\mathrlap{'}_{ i} = \left\{\begin{array}{ll}
(\lambda_{ j,2},\dots,\lambda_{ j,n_{ j}}, k) &  i= j \\
(\lambda_{ k,2},\dots,\lambda_{ k,n_{ k}}, j) &  i= k \\
\lambda_{ i} &  i\neq  j, k \\
\end{array}\right.\]
and the order of rows $ j$, $ k$ in the tableau is reversed. 

\begin{remark}\label{bumpequiv}
A pair of tableaux represent the same combinatorial type if and only if one can be obtained from the other by a sequence of bumps, since cyclic shifts correspond to bumps and independent transpositions do not change the associated tableau.
\end{remark}

For a pair of tableau $\Lambda^1, \Lambda^2$ with respective rows $\lambda^1_{i},\lambda^2_{i}$ appearing in the same order $\rho$ , let $\Lambda^1\cdot \Lambda^2$ denote row-wise concatenation.  That is, $\Lambda^1\cdot \Lambda^2$ has rows \[ \lambda_i = (\lambda^1_{i,1},\dots,\lambda^1_{i,n_{1,i}},\lambda^2_{i,1},\dots,\lambda^2_{i,n_{2,i}}). \] The \df{periodicity} $p$ of a tableau $\Lambda$ is the largest $p$ such that $\Lambda$ is a $p$-fold concatenation of a tableau $\Lambda'$, \[\Lambda = \underset{\times p}{\Lambda' \cdots \Lambda'}.\] We call the tableau $\Lambda'$, the \df{period} of $\Lambda$. We say $\Lambda$ is \df{non-periodic} when $p=1$. The periodicity of a combinatorial type $\coty$ is that of any tableau representing $\coty$, which is well defined by the following lemma. 

\begin{lemma}\label{period}
The period of a tableau is unique, and all tableau representing a fixed combinatorial type have the same periodicity.
\end{lemma}

\begin{proof}
For uniqueness, just observe that row $\lambda\mathrlap{'}_i$ of $\Lambda'$ consists of the first $n_i/p$ entries of row $\lambda_i$ of $\Lambda$ where $n_i$ is the length of $\lambda_i$. Now consider a pair of tableau $\Lambda_1$ and $\Lambda_2 = \bump(\Lambda_1,\{i,j\})$ representing the same $\coty$ with periodicities $p_1$, $p_2$ respectively, and let $\Lambda\mathrlap{'}_1$ be the period of $\Lambda_1$. Then $\Lambda_2$ is the $p_1$-fold concatenation of $\Lambda\mathrlap{'}_2 = \bump(\Lambda\mathrlap{'}_1,\{i,j\})$. Therefore, $p_2 \geq p_1$. Since any tableau representing $\coty$ can be obtained from any other tableau representing $\coty$ by a sequence of bumps, all tableaux have the same periodicity.
\end{proof}

\paragraph{Standard configurations (non-periodic).}

For non-periodic non-layered combinatorial types $\coty$, we will construct a standard set of labeled vector configurations $\wcrs(\coty) \subset \vcrs(\coty)$ parametrize by $\mb S^1$. Here we arbitrarily choose a standard configuration similar to the ``compressed form'' given in \cite[page 31]{knuti}. For layered combinatorial types, we can apply this construction independently to each layer. The periodic case will be dealt with in the next subsection. 

We first construct a labeled vector configuration $W(\coty,\theta,d)$ for $\theta \in \mb S^1$ and $\delta>0$ sufficiently small as follows. Let $\Lambda_\text{min}$ be the lexicographically minimal tableau representing $\coty$ for which there exists exactly one adjacent pair. By Lemma~\ref{lemcontract2} below, such a tableau always exists, provided $\coty$ is non-layered. We will define a sequence of configurations $W_t$ recursively starting from $t=0$. To start, set $\Lambda_0=\Lambda_\text{min}$, $\xi_0=\theta$, $W_0 = \emptyset$. Let $\{(i_{t,1},j_{t,1}),\dots,(i_{t,m_t},j_{t,m_t})\}$ be the set of all adjacent pairs in $\Lambda_t$ ordered according to the rows of $\Lambda_t$. Let 
\[ W_{t+1} = W_t \cup \{(i_{t,1},j_{t,1},\xi_t),\dots,(i_{t,m_t},j_{t,m_t},\xi_t)\}, \] 
$\xi_{t+1} = \xi_t +\delta$, and let $\Lambda_{t+1}$ be the tableau obtained from $\Lambda_t$ by interchanging the corresponding pairs of rows and deleting the first entry from each of these rows. Note that this is similar to a bump except that initial entries are deleted instead of being moved to the end of their respective rows. Eventually, $\Lambda_T = \emptyset$ for some minimal $T$.  Let $W(\coty,\theta,\delta) = W_T$. Finally, let 
\[ \wcrs(\coty) = \{W(\coty,\theta,{2\pi}/{N}): \theta \in \mb S^1 \} .\]

\begin{lemma}\label{lemcontract2}
For any non-layered combinatorial type $\coty$, $\vcrs(\coty)$ is equivariantly homotopic to $\wcrs(\coty) \simeq \mb{S}^1$.
\end{lemma}

\begin{proof}[Proof ($\coty$ non-periodic)] 
For each $V \in \vcrs(\coty)$, we will inductively define a sequence of partitioned vector configurations ${W_k \cup R_k \in \vcrs(\coty)}$ starting from $k=1$, and maps ${\psi_k: \vcrs(\coty) \times [0,1] \to \vcrs(\coty)}$ that perform each step of the induction continuously. Together, these maps will define an equivariant deformation retraction from $\vcrs(\coty)$ to $\wcrs(\coty)$. 

Let $\angle(y,x)$ denote the angular distance from $x$ to $y$ in the counter-clockwise direction, and as in the proof of Lemma \ref{lemcontract1}, let $\delta$ be the minimum angular distance between vectors of $V$ that have non-disjoint labels, so that for all $x,y \in V_i$ we have $0 < \delta \leq \angle(y,x) < 2\pi$. Fix a labeled vector $v\in V$, and set $W_1 =\{v\}$ and $R_1 = V \setminus \{v\}$. We will continuously deform $V$ to a configuration $W(v)$. Later we will make a choice of $v = \tilde v$ that depends continuously on the initial configuration $V \in \vcrs(\coty)$. 

At each step, rotate the sub-configuration $R_k$ clockwise by a continuous rotation ${\phi_{k,t} : [0,1] \to SO(2)}$, $\phi_{k,0} = \text{id}$, that decreases angular distances among the vectors of $W_k \cup \phi_{k,t} R_k$ with non-disjoint labels, until reaching the minimal rotation $\phi_k = \phi_{k,1}$ such that there is some $x \in R_k$ and $y \in W_k$ with non-disjoint labels and $ \angle(\phi_k x, y) = \delta$. Let $X$ be the set of all such $x \in R_k$, let $W_{k+1} = W_k \cup \phi_k X$ and $R_{k+1} = {\phi_k (R_k \setminus X)}$, and continue inductively until $R_K = \emptyset$. Since $\coty$ is \emph{non-layered}, such a pair $x,y$ exists provided $R_k \neq \emptyset$, and since each step removes elements from $R_k$, the process terminates in a configuration $W_K = W_K(v)$ such that for every $x \in W_K \setminus \{v\}$ there exists a $y \in W_K$ such that $\angle(x,y) = \delta$ and $x,y$ have non-disjoint labels; see \fig \ref{compressed}. Let $\psi_{k}(V,t) = W_k \cup \phi_{k,t} R_k$ for $k \leq K$. Since vectors of $\psi_{k}(V,t)$ are only rotated through other vectors with disjoint labels, this process only changes the swap pair of a configuration by elementary operations, which implies $\ct(\mc A_{\psi_{k}(V,t)}) = \ct(\mc A_V)$. Finally, let $\psi_{K+1}(V,t)$ start from $\psi_{K+1}(V,0) = W_K$ and continuously scale the angular distance of each vector from $v$ by sending a labeled vector at $u$ to $u_t$ where $ \angle(u,v) =r\delta$, $\angle(u_t,v) = tr\delta'+(1-t)r\delta$, and $\delta' = 2\pi/N$. Let $W(v) = \psi_{K+1}(V,1)$.

\begin{figure}[h!] \centering
\begin{tikzpicture}[scale=1.4]
\draw (.14,-.2)node[left]{\ft $0$};
\draw (4.74,-.2)node[left]{\ft $2\pi$};
\draw (.95,-.2)node {\ft $v$};
\fill[purple!60, opacity=.4] (.8,0) --++ (0,.3) --++ (.3, 0) --++
(0,-.3) --cycle; 
\fill[gray!60, opacity=.4] (1.9,0) --++ (0,.3) --++ (.3, 0) --++
(0,-.3) --cycle; 
\fill[gray!60, opacity=.4] (3.1,0) --++ (0,.3) --++ (.3, 0) --++
(0,-.3) --cycle; 
\fill[gray!60, opacity=.4] (4.1,0) --++ (0,.3) --++ (.3, 0) --++
(0,-.3) --cycle; 
\fill[gray!60, opacity=.4] (0.2,.3) --++ (0,.3) --++ (.3, 0) --++
(0,-.3) --cycle; 
\fill[gray!60, opacity=.4] (1.3,.3) --++ (0,.3) --++ (.3, 0) --++
(0,-.3) --cycle; 
\fill[gray!60, opacity=.4] (2.7,.3) --++ (0,.3) --++ (.3, 0) --++
(0,-.3) --cycle; 
\fill[gray!60, opacity=.4] (3.5,.3) --++ (0,.3) --++ (.3, 0) --++
(0,-.3) --cycle; 
\fill[gray!60, opacity=.4] (1,.6) --++ (0,.3) --++ (.3, 0) --++
(0,-.3) --cycle; 
\fill[gray!60, opacity=.4] (3.0,.6) --++ (0,.3) --++ (.3, 0) --++
(0,-.3) --cycle; 
\fill[gray!60, opacity=.4] (2.2,.6) --++ (0,.3) --++ (.3, 0) --++
(0,-.3) --cycle; 
\fill[gray!60, opacity=.4] (4.1,.6) --++ (0,.3) --++ (.3, 0) --++
(0,-.3) --cycle; 

\begin{scope}[yscale=.3,xscale=.15]
\draw[thick,black] (0,0) --++ (16/3,0) \us\ue --++ (4/3,0) \us\ue
--++ (12/3,0) \us\ue --++ (10/3,0) \ds\de --++ (4/3,0) \ds\de --++
(6/3,0) \ds\de --++ (2/3,0);  
\draw[thick,cyan] (0,1) --++ (4/3,0) \us\ue --++ (10/3,0) \us\ue
--++ (18/3,0) \ds\de --++ (4/3,0) \ds\de --++ (2/3,0) \ds\de --++
(14/3,0) \us\ue --++ (2/3,0);  
\draw[thick,violet] (0,2) --++ (4/3,0) \ds\de --++ (6/3,0) \ds\de --++
(16/3,0) \us\ue --++ (10/3,0) \us\ue \us\ue --++ (16/3,0) \ds\de --++
(2/3,0);  
\draw[thick,blue] (0,3) --++ (20/3,0) \ds\de \ds\de --++ (6/3,0)
\ds\de --++ (18/3,0) \us\ue --++ 
(2/3,0) \us\ue --++ (6/3,0) \us\ue --++ (2/3,0); 
\end{scope}

\begin{scope}[xshift=-1.5cm]

\draw (7.14,-.2)node[left]{\ft $0$};
\draw (11.74,-.2)node[left]{\ft $2\pi$};
\draw (7.95,-.2)node {\ft $v$};
\fill[purple!60, opacity=.4] (7.8,0) --++ (0,.3) --++ (.3, 0) --++
(0,-.3) --cycle; 
\fill[gray!60, opacity=.4] (8.4,0) --++ (0,.3) --++ (.3, 0) --++
(0,-.3) --cycle; 
\fill[gray!60, opacity=.4] (9,0) --++ (0,.3) --++ (.3, 0) --++ (0,-.3)
--cycle; 
\fill[gray!60, opacity=.4] (9.6,0) --++ (0,.3) --++ (.3, 0) --++
(0,-.3) --cycle; 
\fill[gray!60, opacity=.4] (8.1,.3) --++ (0,.3) --++ (.3, 0) --++
(0,-.3) --cycle; 
\fill[gray!60, opacity=.4] (8.7,.3) --++ (0,.3) --++ (.3, 0) --++
(0,-.3) --cycle; 
\fill[gray!60, opacity=.4] (9.3,.3) --++ (0,.3) --++ (.3, 0) --++
(0,-.3) --cycle; 
\fill[gray!60, opacity=.4] (9.9,.3) --++ (0,.3) --++ (.3, 0) --++
(0,-.3) --cycle; 
\fill[gray!60, opacity=.4] (8.4,.6) --++ (0,.3) --++ (.3, 0) --++
(0,-.3) --cycle; 
\fill[gray!60, opacity=.4] (9,.6) --++ (0,.3) --++ (.3, 0) --++
(0,-.3) --cycle; 
\fill[gray!60, opacity=.4] (9.6,.6) --++ (0,.3) --++ (.3, 0) --++
(0,-.3) --cycle; 
\fill[gray!60, opacity=.4] (10.2,.6) --++ (0,.3) --++ (.3, 0) --++
(0,-.3) --cycle; 

\begin{scope}[xscale=.15,yscale=.3]
\draw[thick,black] (140/3,0) --++ (16/3,0) \us\ue\us\ue\us\ue --++
(2,0) \ds\de\ds\de\ds\de --++ (32/3,0);  
\draw[thick,violet] (140/3,1) --++ (16/3,0) \ds\de --++ (2,0)
\us\ue\us\ue\us\ue --++ (2,0) \ds\de\ds\de --++ (26/3,0);  
\draw[thick,blue] (140/3,2) --++ (22/3,0) \ds\de\ds\de --++ (2,0)
\us\ue\us\ue\us\ue --++ (2,0) \ds\de --++ (20/3,0);  
\draw[thick,cyan] (140/3,3) --++ (28/3,0) \ds\de\ds\de\ds\de --++
(2,0) \us\ue\us\ue\us\ue --++ (20/3,0);  
\end{scope}

\end{scope}

\end{tikzpicture}
\caption{\ft Sending $V$ ({left}) to its compressed form $W_K(v)$ ({right})} 
\label{compressed}
\end{figure}
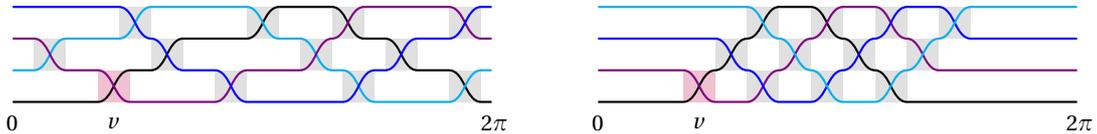

We now define $\tilde v = \tilde v(V)$ in terms of the configurations $\{ W(v) : v \in V\}$. If we rotate $W = W(v)$ by $\phi$ such that that there are no elements of $\phi W$ in the arc $(0,\phi v) \subset \mb S^1$, then the associated tableau $\Lambda$ of $\phi \mc A_W$ will have exactly one adjacent pair, which is given by the support information of $v \in V$. Since $\coty$ is \emph{non-periodic}, there is a unique choice of $v = \tilde v$ such that $\Lambda = \Lambda_\text{min}$.  With this, our final vector configuration becomes $W = W(\tilde v) = W(\coty,\theta,2\pi/N)$, where $\tilde v = (i,j,\theta)$, and the desired equivariant deformation retraction is $\psi_{K+1}\cdots\psi_1$.
\end{proof}

\paragraph{Standard configurations (periodic).}

To define $\wcrs(\coty)$ in the periodic case, set $\Lambda_0$ to be the period of $\Lambda_\text{min}$ and otherwise proceed as in the non-periodic case to obtain a configuration $W_T$ at the minimal $T$ such that $\Lambda_T = \emptyset$. Let 
\[ W(\coty,\theta,d) = \{(i,j, \xi +2\pi k/p):(i,j,\xi) \in W_T, k\in[p]\} \] 
and define $\wcrs(\coty)$ as above. Note that although $\wcrs(\coty)$ is homeomorphic to $\mb S^1$, the map $\mb S^1 \to \wcrs(\coty)$, $\theta \mapsto W(\coty,\theta,{2\pi}/{N})$ is a $p$-to-1 covering. 

\begin{proof}[Proof of Lemma \ref{lemcontract2} ($\coty$ periodic)]
We proceed the same way as in the non-periodic case, except instead of fixing a single vector $v = \tilde v$, we start by fixing all vectors $v_1,\dots,v_p = (i,j,\xi_1),\dots,(i,j,\xi_p) $ corresponding to $\Lambda_\text{min}$.  After continuously rotating certain vectors clockwise and then continuously rescaling the angular distance of each vector from the corresponding $v_i$ as above, we obtain $\psi_{K+1}(V,1) = {W(\coty',\xi_1,\delta') \cup \dots \cup W(\coty',\xi_p,\delta')}$ where $\coty'$ is the combinatorial type represented by the period $\Lambda\mathrlap{'}_\text{min}$ of $\Lambda_\text{min}$ and $\delta' = 2\pi/N$. Note that $\Lambda\mathrlap{'}_\text{min}$ is the lexicographically minimal tableau representing $\coty'$ with exactly one adjacent pair. Now, rotate each subconfiguration $W(\coty',\xi_i,\delta')$ for which $\angle(\xi_i,\xi_{i-1}) > 2\pi/p$ clockwise continuously by $\phi_{i,K+2,t}$ until the vectors $\theta_i = \phi_{i,K+2,1} \xi_i$ are spaced evenly around the circle, to obtain 
\[ \psi_{K+2}(V,1) = W(\coty,\theta_1,\delta') = W(\coty',\theta_1,\delta') \cup \dots \cup W(\coty',\theta_p,\delta') \in \wcrs(\coty) . \qedhere \]
\end{proof}

\paragraph{Contractibility}

\begin{proof}[Proof of Theorem \ref{realizations}] 
In the depth 1 case, the full realization space $\fullrs(\coty)$ is homotopic to the space of support configurations $\vcrs(\coty)$ by Lemma \ref{lemcontract1}, which is homotopic to $\mb{S}^1$ by Lemma \ref{lemcontract2}. Since these homotopies are equivariant, by Remark \ref{circprojrs}, $\projrs(\coty)$ is contractible.   

In the depth $d> 1$ case, partition $\coty$ into layers $\coty = \coty_1 \cup \dots \cup \coty_d$. If we restrict a support configurations of $\coty$ to vectors with labels in a layer $\coty_i$, then we obtain a support configuration of $\coty_i$. Hence, $\vcrs(\coty) \subset \vcrs(\coty_1)\times\dots\times\vcrs(\coty_d)$. In the other direction, if we are given support configurations $V_i \in \vcrs(\coty_i)$, then $\bigcup_{i\in[d]} V_i \in \vcrs(\coty)$. Hence $\vcrs(\coty) = \vcrs(\coty_1)\times\dots\times\vcrs(\coty_d)$, and therefore by Lemmas \ref{lemcontract1} and \ref{lemcontract2}, $\fullrs(\coty)$ is homotopic to a product of $d$ circles, and again by Remark \ref{circprojrs}, $\projrs(\coty)$ is homotopic to a product of $d{-}1$ circles.
\end{proof}

\section{Topological invariance}\label{topinv}

In this section, we prove that combinatorial type is a complete topological invariant of systems of curves on the cylinder. Specifically, we show that the associated tableaux of two systems are related by a sequence of bumps if and only if the systems are related by a self-homeomorphism of the cylinder that preserves orientation and $+\infty$. Theorem \ref{combequivtype} then folows from Remark~\ref{bumpequiv}.

For $0\leq \theta <2\pi$, let $\zeta_\theta = \{\theta\}\times \mb R^1$. We call a curve $\gamma : \mb R \to \mb S^1 \times \mb R^1$ a \df{cut-path} of a system $\mc S$ when $\gamma$ diverges to $\pm\infty$, is oriented from $-\infty$ to $+\infty$, intersects each curve of $\mc S$ exactly once, and intersects each curve one at a time away from any crossings of $\mc S$ and away from $\zeta_0$. We associate to each component $\gamma'$ of $\gamma\setminus\zeta_0$, the \df{region} of the cylinder to the left of $\gamma'$ and bounded by $\zeta_0$. We assume every component of $\gamma\setminus\zeta_0$ intersects some curve of $\mc S$; otherwise we can perform an isotopy of the cylinder that preserves $\mc S$ and removes any components that do not intersect $\mc S$. Let $M = M(\mc S,\gamma)$ denote the sum of the number of curves intersecting each region plus the number of crossings in each region. Note that $M \geq n := |\indx|$, since a cut-path intersects all curves and defines at least one region. 

We define two classes of isotopies of the cylinder, called \df{moves}, sending one system and cut-path to another, while preserving the combinatorial type of the system; see \fig~\ref{cut-moves}.  
\begin{enumerate}[(i)]
\item
If a pair of curves cross each other after $\zeta_0$ before crossing any other curve or the cut-path, then deform the curves to send this crossing through $\zeta_0$ in the clockwise direction.
\item
If a curve crosses the cut-path before intersecting any other curve and the cut-path intersects $\zeta_0$ either (a) immediately after or (b) immediately before this crossing, then deform the cut-path by sending this crossing through $\zeta_0$ in the clockwise direction.  
\end{enumerate}

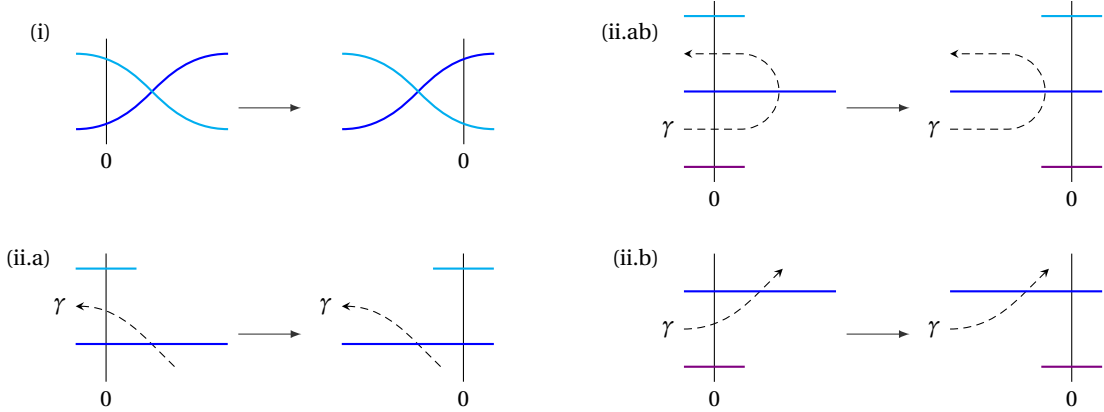
\begin{figure}[h!] \centering
\begin{tikzpicture}

\node (a0) at (0,0) 
{\begin{tikzpicture}
\draw
(.4,-.2) node[below] {\ft 0} -- (.4,1.2)
;
\draw[thick,blue]
(0,0) node[left] {\ft \color{white} $\gamma$}  \us \ue
;
\draw[thick,cyan]
(0,1) \ds \de
;
\end{tikzpicture}}
;
\node[left] at (-1,1) {\ft (i)};

\node (b0) at (3.5,0)
{\begin{tikzpicture}
\draw
(1.6,-.2) node[below] {\ft 0} -- (1.6,1.2)
;
\draw[thick,blue]
(0,0) node[left] {\ft \color{white} $\gamma$}  \us \ue 
;
\draw[thick,cyan]
(0,1) \ds \de 
;
\end{tikzpicture}}
;

\node (a1) at (8,0) 
{\begin{tikzpicture}
\draw
(.4,-.7) node[below] {\ft 0} -- (.4,1.7)
;
\draw[thick,cyan]
(0,1.5) -- (.8,1.5)
;
\draw[thick,violet]
(0,-.5) -- (.8,-.5)
;
\draw[thick,blue]
(0,.5) \g{2}
;
\draw[densely dashed,->,>=stealth]
(0,0) node[left] {\ft $\gamma$} -- (.75,0) arc (-90:90:.5) -- (0,1)
;
\end{tikzpicture}}
;
\node[left] at (7,1) {\ft (ii.ab)};

\node (b1) at (11.5,0)
{\begin{tikzpicture}
\draw
(1.6,-.7) node[below] {\ft 0} -- (1.6,1.7)
;
\draw[thick,cyan]
(1.2,1.5) -- (2,1.5)
;
\draw[thick,violet]
(1.2,-.5) -- (2,-.5)
;
\draw[thick,blue]
(0,.5) \g{2}
;
\draw[densely dashed,->,>=stealth]
(0,0) node[left] {\ft $\gamma$}  -- (.75,0) arc (-90:90:.5) -- (0,1)
;
\end{tikzpicture}}
;

\node (a2) at (0,-3) 
{\begin{tikzpicture}
\draw
(.4,0) node[below] {\ft 0} -- (.4,1.7)
;
\draw[thick,cyan]
(0,1.5) -- (.8,1.5)
;
\draw[densely dashed,<-,>=stealth]
(0,1) node[left] {\ft $\gamma$} \ds \dg{.3}
;
\draw[thick,blue]
(0,.5) \g{2}
;
\end{tikzpicture}}
;
\node[left] at (-1,-2) {\ft (ii.a)};

\node (b2) at (3.5,-3)
{\begin{tikzpicture}
\draw
(1.6,0) node[below] {\ft 0} -- (1.6,1.7)
;
\draw[thick,cyan]
(1.2,1.5) -- (2,1.5)
;
\draw[densely dashed,<-,>=stealth]
(0,1) node[left] {\ft $\gamma$} \ds \dg{.3}
;
\draw[thick,blue]
(0,.5) \g{2}
;
\end{tikzpicture}}
;

\node (a3) at (8,-3) 
{\begin{tikzpicture}
\draw
(.4,-.7) node[below] {\ft 0} -- (.4,1)
;
\draw[thick,violet]
(0,-.5) -- (.8,-.5)
;
\draw[densely dashed,->,>=stealth]
(0,0) node[left] {\ft $\gamma$}  \us \ug{.3}
;
\draw[thick,blue]
(0,.5) \g{2}
;
\end{tikzpicture}}
;
\node[left] at (7,-2) {\ft (ii.b)};

\node (b3) at (11.5,-3)
{\begin{tikzpicture}
\draw
(1.6,-.7) node[below] {\ft 0} -- (1.6,1)
;
\draw[thick,violet]
(1.2,-.5) -- (2,-.5)
;
\draw[densely dashed,->,>=stealth]
(0,0) node[left] {\ft $\gamma$} \us \ug{.3}
;
\draw[thick,blue]
(0,.5) \g{2}
;
\end{tikzpicture}}
;

\foreach \i in {0,...,3}
{
\draw[black!80,->,>=latex]
(a\i) -- (b\i)
;
}

\end{tikzpicture}
\caption{\ft The possible moves on a system and cut-path.} 
\label{cut-moves}
\end{figure}

\begin{remark}
Move (i) removes an crossing from one or more regions and changes the associated tableau of the system by bumping the pair of indices of the curves that cross at that crossing.  While move (ii) decreases the number of curves intersecting some region, possibly deleting that region, the move does not change the associated tableau.  Thus, these moves decrease the value of $M$ and preserve combinatorial type.
\end{remark}

\begin{proof}[Proof of Theorem \ref{combequivtype}]
By Theorem \ref{realizations} for any pair of systems $\mc S$, $\mc T$ of the same combinatorial type, there is a path in $\fullrs(\coty)$ from $\mc S$ to $\mc T$, so there is an isotopy sending $\mc S$ to $\mc T$.

For the other direction, suppose we are given a pair of systems $\mc S$, $\mc T$ on $\indx$ and an orientation preserving homeomorphism $\varphi: \mb S^1 \times \mb R^1 \to \mb S^1 \times \mb R^1$ that preserves $+\infty$. We will find a sequence of bumps sending the associated tableau of $\mc T$ to that of $\mc S$. Choose some $\eps >0$ sufficiently small so that $M(\mc S,\zeta_\eps) = n$.  That is, $\zeta_\eps$ is a cut-path for $\mc S$ with a single region that does not contain any crossings. And, choose $\eps$ generically so that $\eta := \varphi (\zeta_0)$ is a cut-path of $\mc T$. We will perform a sequence of moves starting from $(\mc{T},\eta)$.  

Since each move decreases the value of $M$ and $M \geq n$, we can perform a sequence of moves on $(\mc T,\eta)$ until no more moves are possible. If some region contains an crossing then we can perform move (i), and if there are multiple regions then we can perform move (ii). Therefore, we obtain a system $\mc U$ with associated tableau $\Lambda$ and cut-path $\gamma$ with a single region that contains no crossings. 
The local sequence of a curve $U_i\in \mc U$ is given by the order in which $U_i$ crosses the other curves after crossing $\gamma$, which is the same as that of $S_i \in \mc S$. 
The rows of $\Lambda$ are the local sequences of $\mc U$ in the order $\gamma$ crosses the curves of $\mc U$, which is the same as order as that of $\mc S$.  Hence $\mc U$ and $\mc S$ have the same the associated tableau $\Lambda$.  Furthermore, $\Lambda$ is obtained from the associated tableau of $\mc T$ by preforming the sequence of bumps corresponding to the above sequence of moves. Therefore $\mc U$, $\mc T$, and $\mc S$ all have the same combinatorial type. 
\end{proof}

\section{Universality} \label{universe}

In this section, we prove a universality theorem for arrangements of $k$-gons.
We actually prove the following slightly more specific result.

\begin{lemma}\label{prodrs}
For any $k$ order types $\chi_1, \dots, \chi_k$ on $[n]$, where at least two are distinct, there is a combinatorial type $\coty$ on $[n]$ such that its $k$-gon realization space $\projrs_k(\coty)$ is homotopy equivalent to $\projrs_1(\chi_1) \times \cdots \times \projrs_1(\chi_k)$.
\end{lemma}

Theorem~\ref{kgonrs} follows immediately.

\begin{proof}[Proof of Theorem~\ref{kgonrs}]
Fix a basic primary semialgebraic set $\mc Z$ and $k >1$.  Let $\chi_1$ be the order type of the Mnëv point set with point realization space homotopic to $\mc Z$. Let $\chi_2,\dots,\chi_k$ all be the order type of $n$ points in convex position. Note that the point realization space of $n$ points in convex position is contractible. With this, the $k$-gon realization space of $\coty$ from Lemma~\ref{prodrs} is also homotopic to $\mc Z$.
\end{proof}

To show Lemma~\ref{prodrs}, we construct a combinatorial type $\coty$ such that for every realization $\mc{A}$ of $\coty$ by $k$-gons, the vertices of each $k$-gon can be labeled.  That is, each vertex can be uniquely identified using only information encoded in the combinatorial type.  Note that this is not possible in general, as combinatorial type does not provide information about individual vertices directly.  Furthermore, we construct $\coty$ so that the order type of the vertices of $\mc{A}$ is the same in every realization and each $\chi_i$ appears as a subset of the vertices.

We define $\coty$ in two ways: in the primal we construct an arrangement of $k$-gons, 
then in the dual we construct a system of curves. 
We show in Lemma~\ref{univbodiestype} that these two constructions define the same combinatorial type. We will use both the primal and the dual construction in the proof of Lemma~\ref{prodrs}. 

Before defining $\coty$, index the order types $\chi_i$ so that the cyclic ordering $\chi_1,\chi_2,\dots,\chi_k,\chi_1,\dots$ is not periodic with period smaller than $k$.  This is possible by the assumption that there are at least two distinct order types.

\paragraph{The primal construction.}

The primal construction $\mc A$ depends on certain arbitrary choices that will not affect the combinatorial type. Assume for the primal construction that each $\chi_i$ is realizable; the non-realizable case is defined by the dual construction only.

Let $\mc A_0$ be an arrangement of $2k$ points in convex position denoted by $a_1^1\cm a_1^n\cm a_2^1\cm a_2^n,\dots\cm a_k^1\cm a_k^n$ in counter-clockwise order, such that the lines $\ell_i$ spanning $a_i^n$ and $a_{i+1}^1$ bound a convex \hbox{$k$-gon} $B$.\footnote{%
Here subscripts are indices over $\mb{Z}_k$, so in particular $\ell_k$ is the line spanning $a_k^n$ and $a_1^1$.}
Observe that $B \setminus \conv(\mc{A}_0)$ consists of $k$ triangular regions.  We construct $\mc{A}$ by placing a point set realizing one of the $\chi_i$ in each of these traingular regions, then we define the $k$-gons $A^s$ to have vertices consisting of one point from each region;
see \fig~\ref{univprimal} for an example with $n=6$, $k=4$.

Let $\chi_i$ be defined on the index set $\{({}_i^1),\dots,({}_i^n)\}$, and  
let $\mc P_i = \{p_i^1,\dots,p_i^n\}$ be a realization of $\chi_i$.  
Furthermore, let $\chi_i$ indexed so that $p_i^1$ and $p_i^2$ appear on the boundary of the convex hull of $\mc P_i$ and the local sequence of $p_i^{1}$ is $p_i^2,p_i^3,\dots,p_i^n$.  That is, the angles $\theta_i^{s}$ at $p_i^1$ from the semiline through $p_i^2$ to the semiline through $p_i^s$ are increasing in the counter-clockwise direction, $0 = \theta_i^2 < \theta_i^3< \dots< \theta_i^n < \pi$.  Note that this implies $p_i^n$ is also on the boundary of the convex hull of $\mc P_i$, which we will call the \df{convex boundary} for short.
Now augment $\mc P_i$ by two points as follows.  Let $\mc Q_i = \mc P_i \cup \{q_i^1,q_i^n\}$ such that 
${p_i^n\cm q_i^1\cm q_i^n\cm p_i^1}$ appear consecutively in counter-clockwise order on the convex boundary of $\mc Q_i$ 
and no line through any two points of $\mc P_i$ separates the points $q_i^1, q_i^n, p_i^1$. 
Note that this uniquely determines 
the order type of $\mc Q_i$; see \fig~\ref{univprimal} (left).

\begin{figure}[h!] \centering
\begin{tikzpicture}[scale = .69]
\begin{scope}[rotate=8]
\fill[black!10!white] (-1,4) -- (-2.4,1.2) --(0,0) --cycle;
\fill[black!10!white] (1.5,9) -- (0,6) --(3,8) --cycle;
\fill[black!10!white] (84/11,54/11) -- (6,6) --(7,3) --cycle;
\fill[black!10!white] (36/7,-18/7) -- (6,0) --(2,-1) --cycle;

\draw[thin, black!30] 
(-3,3/2) --++ (9,-4.5)
(-3,0) --++ (5,10)
(5,-3) --++ (3,9)
(0,10) --++ (9,-6)
;

\fill (-1.4,2.4) coordinate (a1'2) circle [radius =.055];
\fill (-1.1,2) coordinate (a1'3) circle [radius =.055];
\fill (-1,1) coordinate (a1'4) circle [radius =.055];
\fill (-0.7,0.75) coordinate (a1'5) circle [radius =.055];
\fill (5.3,-1.4) coordinate (a2'4) circle [radius =.055];
\fill (4.1,-1.8) coordinate (a2'2) circle [radius =.055];
\fill (4.5,-1.6) coordinate (a2'3) circle [radius =.055];
\fill (5.6,-0.75) coordinate (a2'5) circle [radius =.055];
\fill (6.4,5.2) coordinate (a3'5) circle [radius =.055];
\fill (6.7,4.6) coordinate (a3'4) circle [radius =.055];
\fill (7,4.3) coordinate (a3'3) circle [radius =.055];
\fill (7.1,3.9) coordinate (a3'2) circle [radius =.055];
\fill (1.5,7.7) coordinate (a4'3) circle [radius =.055];
\fill (2.2,8) coordinate (a4'2) circle [radius =.055];
\fill (0.8,7.1) coordinate (a4'5) circle [radius =.055];
\fill (1.2,7.3) coordinate (a4'4) circle [radius =.055];

\foreach \k in {2,3,4,5}
{\draw[thin,gray]
(a4'\k)
\foreach \j in {1,2,3,4}
{ -- (a\j'\k)};}

\foreach \k in {2,3,4,5}
{
\foreach \j in {1,2,3,4}
{
\fill (a\j'\k) circle [radius =.055]
;}}


\draw[thick, violet] (0,0) --++ (6,0) --++(0,6)--++ (-6,0) --cycle; 
\draw[very thick, blue] (2,-1) -- (7,3) -- (3,8) -- (-1,4) --cycle;

\draw (-1.1,1.4) node[left]{\ft $\chi_1$};
\draw  (5.3,-2) node[left]{\ft $\chi_2$};
\draw (7.6,4.7)node[left]{\ft $\chi_3$};
\draw (2.1,8.2)node[left]{\ft $\chi_4$};
\draw (0,-.3)node[left]{\ft $a^6_1$};
\draw (6.9,-.1)node[left]{\ft $a^6_2$};
\draw (6.9,6.1)node[left]{\ft $a^6_3$};
\draw (0,6)node[left]{\ft $a^6_4$};
\draw (-1,4.1)node[left]{\ft $a^1_1$};
\draw (2,-1.2)node[left]{\ft $a^1_2$};
\draw (7.1,2.9)node[right]{\ft $a^1_3$};
\draw (3.1,8.1)node[right]{\ft $a^1_4$};
\fill[violet] (0,0) circle [radius = .06];
\fill[violet] (6,0) circle [radius = .06];
\fill[violet] (6,6) circle [radius = .06];
\fill[violet] (0,6) circle [radius = .06];
\fill[blue] (2,-1) circle [radius = .06];
\fill[blue] (7,3) circle [radius = .06];
\fill[blue] (3,8) circle [radius = .06];
\fill[blue] (-1,4) circle [radius = .06];
\end{scope}

\begin{scope}[xshift = -6.6cm, yshift = 4.5cm]
\draw[thin, black!20] 
(0.2,-0.8) -- (-1.2,4.8)
(0.4,-.4) -- (-4.3,4.3)
(1/4,-3/16) -- (-5,15/4)
(0.4,-0.2) -- (-5.6,2.8)
(-6,0) -- (.6,0)
(-4,-2.8) -- (-4,4.7)
(-5.5,2.5) -- (-.4,4.2)
(-4.4,-2.8) -- (-.5,5)
(-11/2,-3/2) -- (-7/2,9/2)
(-6,2) -- (.6,2)
(-17/3,-4/3) -- (-3,4)
(-6,-2/3) -- (0,10/3)
(-11/2,-1/2) -- (0,5)
(-5,4/3) -- (0,14/3)
(-5,7/2) -- (1/2,3/4)
;
\draw[thick, violet] (-1.3,-1.6) --(-5,0);
\draw[very thick, blue] (-3,-5/3) --(0,0);
\draw (0,.3) node[right]{\ft $p^1_1$};
\draw (-1,4.1) node[right]{\ft $p^2_1$};
\draw (-2,2) node[right]{\ft $p^3_1$};
\draw (-4,3) node[left]{\ft $p^4_1$};
\draw (-4,2) node[left]{\ft $p^5_1$};
\draw (-5,0.2) node[left]{\ft $p^6_1$};
\draw (-1.3,-1.6) node[right]{\ft $q^6_1$};
\draw (-3,-5/3) node[left]{\ft $q^1_1$};
\draw[fill] (-1,4) circle [radius =.055];
\draw[fill] (-2,2) circle [radius =.055];
\draw[fill] (-4,3) circle [radius =.055];
\draw[fill] (-4,2) circle [radius =.055];
\fill[blue] (0,0) circle [radius =.06];
\fill[violet] (-5,0) circle [radius =.06];
\fill[blue] (-3,-5/3) circle [radius =.06];
\fill[violet] (-1.3,-1.6) circle [radius =.06];
\draw[<-, >= latex, black!50] (4.2,-.5) arc [radius=6, start
angle=62, end angle= 95]; 
\draw[->, >= latex, black!50] (-2.8,-1.9) arc 
[radius=13, start angle=225, end angle= 276];
\draw[->, >= latex, black!50] (-0.45,-1.6) arc 
[radius=11, start angle=281, end angle= 312];
\draw[->, >= latex, black!50] (-5,-.4) arc 
[radius=8, start angle=200, end angle= 293];
\end{scope}
\end{tikzpicture}
\caption{\ft The point set $\mc P_1$ on the left is
  mapped to points on the
  right by the projective transformation determined by 
$p^1_1\mapsto a^1_1,\  q^1_1 \mapsto a^1_{2},\ p^6_1 \mapsto a^6_1,\ q^6_1 \mapsto a^6_4 $}  \label{univprimal}
\end{figure}
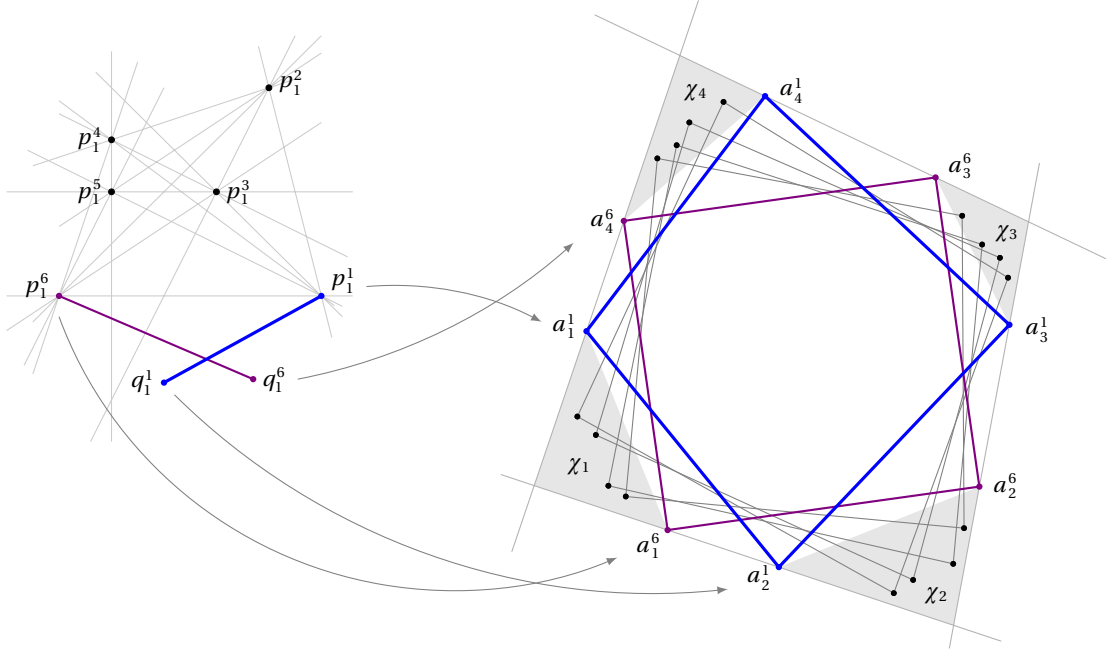

Now define projective transformations $\phi_i$ such that \[ \phi_i(q_i^n) = a_{i-1}^n, \quad \phi_i(p_i^1) = a_i^1, \quad \phi_i(p_i^n) = a_i^n, \quad \phi_i(q_i^1) = a_{i+1}^1 , \] and let $\mc P = \{a_1^1,a_1^2,\dots,a_2^1,\dots,a_n^n\}$ where $a_i^s = \phi_i(p_i^s)$. Note that $\phi_i$ sends $\mc Q_i$ into $B$ and preserves the orientation of $q_i^1, q_i^n, p_i^1$, so $\mc Q_i \projeq \{a_i^1,\dots,a_i^n,a_{i-1}^n,a_{i+1}^1\}$ with appropriate relabeling. Let $\mc A = \{A^1,\dots,A^{n}\}$  where $A^s = \conv(\{a_1^s,a_2^s,\dots,a_k^s\})$, and 

Finally, let $\coty$ denote the combinatorial type of $\mc{A}$, let $\chi$ denote the order type of $\mc P$, and let $\psi_i$ denote the order type of $\mc{Q}_i$.

\FloatBarrier

\subsection{Path systems}

We call the graph of some indexed family of functions defined over an interval a \df{path system}. 
We say two path systems are \df{equivalent} when the are related by an orientation preserving self-homeomorphism of the plane that also preserves indices and the orientations of the paths. 
We will always assume that the end-points of a path system are all distinct, and 
that the paths satisfy the same genericity conditions given in Subsection~\ref{sect:generic} for systems of curves. 
For path systems $\mc L_1,\mc L_2$ over intervals $I_1,I_2 \subset \mb R$ with the same number of paths, the concatenation $\mc L_1 \cdot \mc L_2$ is the path system obtained by identifying the right edge of $I_1 \times \mb R$ with the left edge of $I_2 \times \mb R$ by a homeomorphism sending the right end-points of $\mc L_1$ to the left end-points of $\mc L_2$.  Here indices must be dealt with appropriately. 
If the left end-points of $\mc L_1$ appear in the same order as the right end-points, then we may form a system of curves $\cyclecat \mc L_1$ by identifying the left and right edges of $I_1 \times \mb R$ by a homeomorphism that identifies the left and right end-points of each path in $\mc L_1$. 
Let $\vertflip \mc L_1$ denote the path system obtained by flipping $\mc L_1$ vertically by the map $(x,y) \mapsto (x,-y)$.  
Given an order type $\chi$, we say a path system $\mc L$ is a \df{pseudoline representation} of $\chi$ when $\mc S = {\cyclecat (\mc L \cdot \vertflip \mc L)}$ is an orientable system with order type $\chi$ as in Theorem~\ref{orientable}.  
Note that classes of equivalent path pseudoline representations of $\chi$ bijectively correspond to tableau representations of $\chi$. 
We say an element $i$ is on the \df{convex boundary} of $\chi$ when the corresponding curve $S_i$ appears on the upper envelope of a corresponding system $\mc S$. 

\begin{remark}\label{firstpath}
For each element $i$ on the convex boundary of an order type $\chi$, there is a unique class of equivalent pseudoline representations $\mc L$ where $L_i$ starts as the top most path and crosses all other paths, thereby going to the bottom, before any other crossings occur. 
\end{remark}

\subsection{The dual construction}

Let $\chi_i$ be an order type on elements $\{({}_i^1),\dots,({}_i^n)\}$ indexed as in the primal construction, and 
let $\mc L_i$ be a pseudoline representation of $\chi_i$ with paths $L_i^1,\dots,L_i^n$ such that $L_i^1$ starts at the top and crosses all other paths first as in Remark \ref{firstpath}. 
Let $\mc C = \{C_1,\dots,C_k\}$ be the dual system of $k$ points in convex position indexed in counter-clockwise order, and observe that each curve $C_i$ appears exactly once on the upper envelope and once on the lower envelope of $\mc C$.
Let $\mc S$ be a system of curves where each curve $C_i \in \mc C$ is replaced by $n$ curves $\{S^1_i,\dots,S^n_i\}$ in a small tubular neighborhood about $C_i$ crossing to form a copy of $\mc L_i$ above all other curves of $\mc S$ and a copy of $\vertflip \mc L_i$ below all other curves of $\mc S$. 
Let $T^s$ be the upper envelope of the curves $S_1^s,\dots,S_k^s$, and let $\mc T = \{T^1,\dots,T^n\}$.
Equivalently, 
let $\mc U$ be the path system of size $n$ where each path from bottom to top crosses all paths below itself (beginning with the bottom path crossing no other paths and ending with the top path crossing all other paths), and let $\mc T= \cyclecat(\mc L_1\cdot \mc U\cdot \mc L_2\cdot \mc U \cdots \mc L_k\cdot \mc U)$. 
See \fig~\ref{univdual} for an example with $n=6$, $k=4$. 

Finally, let $\coty$ denote the combinatorial type of $\mc{T}$ and let $\chi$ denote the order type of $\mc{S}$.

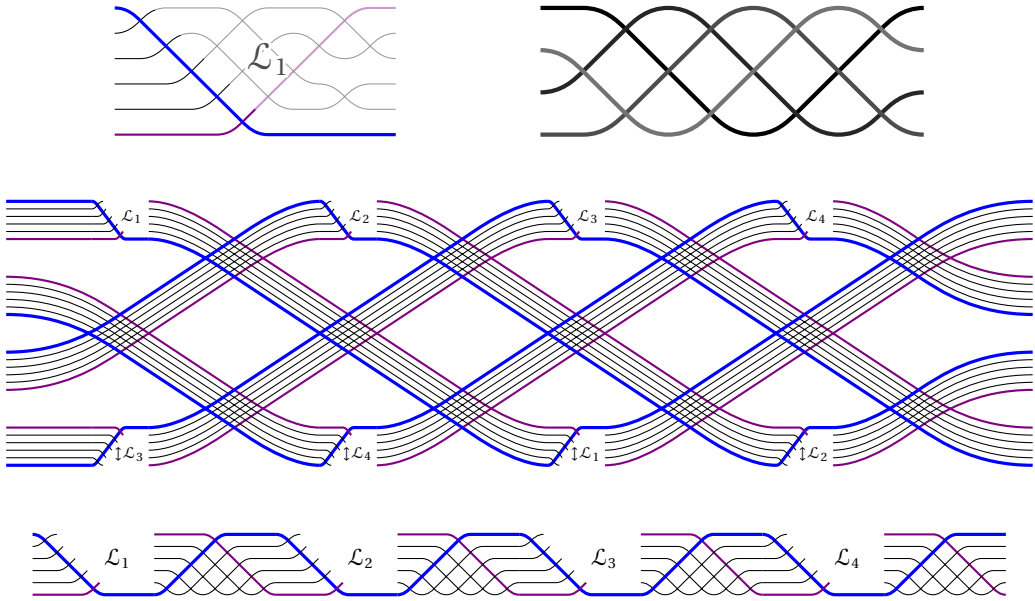
\begin{figure}[htp!]

\centering

\begin{tikzpicture}[scale=.8]

\begin{scope}[xshift=-7cm,yscale=-.42,xscale=.42,yshift=-6cm]

\draw[thick,violet]
(0,6) \g{4} \ds \dg{4} \de \g{1}
;

\draw
(0,5) \g{3} \ds \dg{2} \de \us \ue \ds \de
(0,4) \g{2} \ds \dg{2} \de \g{2} \us \ug{1} \ue
(0,3) \g{1} \ds \de \us \ug{2} \ue \g{1} \ds \de \g{1}
(0,2) \ds \de \g{2} \us \ug{2} \ue \us \ue \g{1}
;

\draw[very thick, blue]
(0,1) \us \ug{4} \ue \g{5}
;

\fill[white,opacity=.6]
(1,.5) -- ++(5,5) -- ++(5.5,0) -- ++(0,-5) -- cycle  
;

\node[black!70] at (6,3) {\contour{white}{\Large $\mc L_1$}};

\end{scope}

\begin{scope}[scale=.7,yshift=-1cm]

\draw[ultra thick, black]
(0,4) \g{1} \ds \dg{2} \de \us \ug{2} \ue 
;
\draw[ultra thick, black!85]
(0,2) \us \ug{1} \ue \ds \dg{2} \de \us \ue 
;
\draw[ultra thick, black!70]
(0,1) \g{1} \us \ug{2} \ue \ds \dg{2} \de
;
\draw[ultra thick, black!55]
(0,3) \ds \dg{1} \de \us \ug{2} \ue \ds \de
;

\end{scope}

\end{tikzpicture}

\vspace{.8cm}

\begin{tikzpicture}[xscale=1.5]

\draw[thick,violet]
(0,4) \g{.75} coordinate (a1_0_1) ++(.5,.5)  \dsx \dg{2} \dex coordinate (a1_0_2) ++(.5,-.5) \usx \ug{2} \ue 
(0,2) \us \ug{1} \uex coordinate (a2_0_1) ++(.5,.5) \dsx \dg{2} \dex coordinate (a2_0_2) ++(.5,-.5) \usx \ue 
(0,1.5) \g{.75} coordinate (a3_0_2) ++(.5,-.5) \usx \ug{2} \uex coordinate (a3_0_1) ++(.5,.5) \dsx \dg{2} \de
(0,3.5) \ds \dg{1} \dex coordinate (a4_0_2) ++(.5,-.5) \usx \ug{2} \uex coordinate (a4_0_1) ++(.5,.5) \dsx \de 
;
\foreach \x in {1,2,3,4}
{
\draw[thin]
(0,{4+.1*\x}) \g{.75} coordinate (a1_\x_1) ++(.5,0) \dsx \dg{2} \dex coordinate (a1_\x_2) ++(.5,0) \usx \ug{2} \ue
(0,{2+.1*\x}) \us \ug{1} \uex coordinate (a2_\x_1) ++(.5,0) \dsx \dg{2} \dex coordinate (a2_\x_2) ++(.5,0) \usx \ue 
(0,{1+.1*\x}) \g{.75} coordinate (a3_\x_2) ++(.5,0) \usx \ug{2} \uex coordinate (a3_\x_1) ++(.5,0) \dsx \dg{2} \de 
(0,{3+.1*\x}) \ds \dg{1} \dex coordinate (a4_\x_2) ++(.5,0) \usx \ug{2} \uex coordinate (a4_\x_1) ++(.5,0) \dsx \de  
;
}
\draw[very thick,blue]
(0,4.5) \g{.75} coordinate (a1_5_1) ++(.5,-.5) \dsx \dg{2} \dex coordinate (a1_5_2) ++(.5,.5) \usx \ug{2} \ue 
(0,2.5) \us \ug{1} \uex coordinate (a2_5_1) ++(.5,-.5) \dsx \dg{2} \dex coordinate (a2_5_2) ++(.5,.5) \usx \ue 
(0,1) \g{.75} coordinate (a3_5_2) ++(.5,.5) \usx \ug{2} \uex coordinate (a3_5_1) ++(.5,-.5) \dsx \dg{2} \de
(0,3) \ds \dg{1} \dex coordinate (a4_5_2) ++(.5,.5) \usx \ug{2} \uex coordinate (a4_5_1) ++(.5,-.5) \dsx \de 
;

\begin{scope}[yscale=.1,xscale=.05]
\foreach \k in {1,2,3,4}
{
\draw[very thick,blue]
(a\k_5_2) \us \ug{4} \ue \g{4} 
;
\draw[thin]
(a\k_1_2) \ds \de
(a\k_2_2) \g{1} \ds \dg{.5}
(a\k_3_2) \g{2} \ds \dg{.5}
(a\k_4_2) \g{3} \ds \dg{.5}
;
\draw[thick,violet]
(a\k_0_2) \g{4} \ds \dg{.5} 
;
\path (a\k_2_2) ++(6.6,-.5) node {\tiny $\vertflip\! \mc L_{\k}$}
;
}

\foreach \k in {1,2,3,4}
{
\draw[very thick,blue]
(a\k_5_1) \ds \dg{4} \de \g{4} 
;
\draw[thin]
(a\k_4_1) \us \ue
(a\k_3_1) \g{1} \us \ug{.5}
(a\k_2_1) \g{2} \us \ug{.5}
(a\k_1_1) \g{3} \us \ug{.5}
;
\draw[thick,violet]
(a\k_0_1) \g{4} \us \ug{.5} 
;
\path (a\k_3_1) ++(7,0) node {\tiny $\mc L_{\k}$}
;
}
\end{scope}

\end{tikzpicture}

\vspace{.8cm}

\begin{tikzpicture}[scale=.16]

\foreach \x in {}
{
\draw[very thick, blue]
(\x,6) \g{1}
;
\draw
(\x,5) \g{1} (\x,4) \g{1} (\x,3) \g{1} (\x,2) \g{1}
;
\draw[thick,violet]
(\x,1) \g{1}
;
}

\foreach \x in {}
{
\draw[very thick, blue]
(\x,1) \g{1}
;
\draw
(\x,5) \g{1} (\x,4) \g{1} (\x,3) \g{1} (\x,2) \g{1}
;
\draw[thick,violet]
(\x,6) \g{1}
;
}

\foreach \x in {0,20,40,60}
{
\draw[very thick,blue]
(\x,6) \ds \dg{4} \de \g{4} 
;
\draw[thin]
(\x,5) \us \ue
(\x,4) \g{1} \us \ug{.5}
(\x,3) \g{2} \us \ug{.5}
(\x,2) \g{3} \us \ug{.5}
;
\draw[thick,violet]
(\x,1) \g{4} \us \ug{.5} 
;
}

\foreach \x in {10,30,50,70}
{
\draw[thick,violet]
(\x,6) \g{4} \ds \dg{4} \de
;
\draw
(\x,5) \g{3} \ds \dg{3} \de \us \ue 
(\x,4) \g{2} \ds \dg{2} \de \us \ug{1} \ue \g{1}
(\x,3) \g{1} \ds \dg{1} \de \us \ug{2} \ue \g{2}
(\x,2) \ds \de \us \ug{3} \ue \g{3}
;
\draw[very thick, blue]
(\x,1) \us \ug{4} \ue \g{4}
;
}

\path
(07,4) node {\ft $\mc L_1$}
(27,4) node {\ft $\mc L_2$}
(47,4) node {\ft $\mc L_3$}
(67,4) node {\ft $\mc L_4$}
;

\end{tikzpicture}

\caption{%
\textbf{Top left:} The pseudoline representation $\mc L_1$ of $\chi_1$. 
\textbf{Top right:} The system $\mc C$. \newline
\textbf{Center:} The system $\mc S$.
\textbf{Bottom:} The system $\mc T$ of combinatorial type $\coty$.
}\label{univdual}
\end{figure}

\begin{lemma}\label{univpointstype}
$\mc P^*$ and $\mc S$ are orientable and have the same well defined combinatorial type $\chi$.
\end{lemma}

\begin{proof}
Consider the orientations of a triple $a_u^r,a_v^s,a_w^t \in \mc{P}$, and the corresponding curves $S_u^r\cm S_v^s\cm S_w^t \in \mc S$.  If $u,v,w$ are distinct, then these points have the same orientation as $a_u^1,a_v^1,a_w^1$, since each $a_i^x$ among these points is between $a_i^1$ and $a_i^n$ in the local sequence of $a_{i-1}^n$ among $\phi_i(\mc Q_i)$, which implies $a_i^x$ is in convex position together with $\mc A_0$ between $a_i^1$ and $a_i^n$ in counter-clockwise order.  Furthermore, the curves have the same orientation as $C_u,C_v,C_w$, which is the same as that of $a_u^1\cm a_v^1\cm a_w^1$, since each $S_i^x$ is in a small tubular neighborhood about $C_i$. If $u=v=w$ then the orientation of $a_u^r\cm a_u^s\cm a_u^t$ as well as that of $S_u^r,S_u^s,S_u^t$ is determined by $\chi_u$.  If $u=v$ and $w$ is distinct, then both the points and the curves have the same orientation as $p_u^r,p_u^s,p_u^1$. In any case, each triple of $\mc S$ is orientable and its orientation is fixed and is the same as that of $\mc P^*$, and since order types are completely determined by the orientation of each triple, this determines the order type of $\mc P$ and $\mc S$. 
\end{proof}

\begin{lemma}\label{univbodiestype}
$\mc A^*$ and $\mc T$ have the same well defined combinatorial type $\coty$.
\end{lemma}

\begin{proof}
This follows from Lemma~\ref{univpointstype} and the fact that the support function of a polygon is the upper envelope of the support functions of its vertices.
\end{proof}

\paragraph{\textit{Proof of Lemma~\ref{prodrs}}.}

First observe that $\fullrs_1(\chi_i)$ is homotopic to $\fullrs_1(\psi_i)$, since $\psi_i$ is reducible by $q_i^1$ and $\omega_i := \psi_i \setminus q_i^1$ is reducible by $q_i^n$ \cite[Lemma 8.2.1]{OMS}.  That is, for any realization $\mc P_i$ of $\chi_i$, there is a non-empty convex region where $\mc P_i$ can be augmented by a point $q_i^n$ to obtain a realization of $\omega_i$, and the fibers of the deletion map $\delta : \fullrs_1(\omega_i) \to \fullrs_1(\chi_i)$ defined by deleting the point $q_i^n$ are given by this convex region, which implies $\fullrs_1(\omega_i)$ and $\fullrs_1(\chi_i)$ are homotopic. Likewise, $\fullrs_1(\psi_i)$ and $\fullrs_1(\omega_i)$ are homotopic.

Let $\omega$ denote the non-simple order type of the point set $\mc A_0 \cup \{b_1,\dots,b_k\}$ where $b_i$ is the meet of lines $\ell_{i-1},\ell_i$.
We claim that $\fullrs_1(\chi)$ is homeomorphic to $\fullrs_1(\omega) \times \projrs_1(\psi_1) \times \dots \times \projrs_1(\psi_k)$.
To see this, let
\[ \pi_0 : \fullrs_1(\chi) \to \fullrs_1(\omega) \quad \text{and} \quad \pi_i : \fullrs_1(\chi) \to \projrs_1(\psi_i)\]
where $\pi_0$ is defined by restricting to the points $a_1^1\cm a_1^n\cm b_1,\dots\cm a_n^1\cm a_n^n\cm b_n$,
and $\pi_i$ is defined by restricting to the points $a_i^1,\dots,a_i^n,a_{i-1}^n,a_{i+1}^1$ respectively relabeled as $p_i^1,\dots,p_i^n,q_i^n,q_i^1$ and identifying realizations that are projectively equivalent. Let $\pi = \pi_0 \times\dots\times \pi_k$.  The construction of $\mc P$ in the primal definition of $\coty$ shows that $\pi$ is surjective, and $\pi$  
is injective since the points of $\mc Q_i$ are uniquely determined by the projective equivalence class to which they belong and the positions of the points $q_i^n\cm p_i^1\cm p_i^n\cm q_i^1$, which are fixed by a realization of $\omega$. Hence $\pi$ is a bijection, and since it and its inverse are continuous, the claim holds. 

Next, we claim $\fullrs_k(\coty)$ is homeomorphic to $\fullrs_1(\chi)$. Let $\varphi : \fullrs_1(\chi) \to \fullrs_k(\coty)$ by taking convex hulls as in the primal definition of $\coty$.  Assume $\fullrs_k(\coty)$ is non-empty and consider $\mc A = \{A^1,\dots,A^n\} \in \fullrs_k(\coty)$; the existence of $\varphi$ implies that if $\fullrs_k(\coty)$ where empty then $\fullrs_1(\chi)$ would also be empty.  We will show that there is a unique way of indexing the vertices of each body $A^s$ by $\{a_1^s,\dots,a_n^s\}$ so that together they realize $\chi$, which implies there is the unique point set $\mc V \in \fullrs_1(\chi)$ for which $\varphi(\mc V) = \mc A$.

We will first see that such an indexing exists, which implies $\varphi$ is surjective.  Notice that $A^{1*}$ and $A^{t*}$ cross $2k$ times for each $t \neq 1$, so $A^1$ and $A^t$ each appear $k$ times on the boundary of $\conv(A^1 \cup A^t)$, which implies each $A^s$ must be a $k$-gon and the vertices of $A^1$ and $A^t$ are in convex position.  First choose some indexing of the vertices of $A^1$ by $v_1^1,\dots,v_k^1$ in counter-clockwise order, and then index the vertex of each $A^t$ between $v_i^1$ and $v_{i+1}^1$ by $v_i^t$, and let $\mc V_i = \{v_i^1,\dots,v_i^n\}$.  Now let $\mu_i^t$ be the outward normal direction of the common supporting tangent of $A^1$ and $A^t$ through $v_i^1$ and $v_i^t$, and let $\xi_i^t$ be the outward normal direction of the common supporting tangent through $v_i^t$ and $v_{i+1}^1$.  The support curve of the point $v_i^t$ coincides with $A^{t*}$ on the half open interval $[\mu_i^t,\xi_i^t) \subset \mb S^1$, and by Lemma~\ref{univbodiestype}, has an crossing with $A^{s*}$ in $[\mu_i^s,\xi_i^s)$ corresponding to a common supporting tangent through $a_i^t, a_i^s$ for $t,s$ distinct.  This fixes a pseudoline representation for each $\mc V_i$, which is equivalent to $\mc P^*_{i-h}$ for an appropriate cyclic shift of indices by $h \in \mb Z_k$, so $\mc P_i := \mc V_{i+h} \in \fullrs_1(\chi_{i})$. If we perform the primal construction by choosing the point sets $\mc P_i$ as realizations of $\chi_i$, and choosing $q_i^1,q_i^n,\mc A_0$ among the vertices of $A^1$ and $A^n$ so that each map $\phi_i$ is the identity, then we obtain the same arrangement $\mc A$ that we started with.  Thus, by Lemma~\ref{univpointstype} the vertices of $\mc A$ labeled by $a_i^s = v_{i+h}^s$ have order type $\chi$.

We will now see this indexing is unique, which implies $\varphi$ is injective.  Since the vertices of $A^1$ and $A^t$ appear in an alternating order around the convex boundary of their union, the indexing of the vertices is fixed up to a cyclic shift, and since the cyclic ordering of $\chi_1,\dots,\chi_k$ does not have periodicity smaller than $k$, $h$ is the unique cyclic shift of indices for which the vertices have order type $\chi$.  Hence $\varphi$ is a bijection, and since it and its inverse are continuous, the claim holds.

Finally, $\fullrs_k(\coty)$ is homeomorphic to $\fullrs_1(\chi)$, which is homeomorphic to $\fullrs_1(\omega) \times \projrs_1(\psi_1) \times \dots \times \projrs_1(\psi_k)$, so by identifying projectively equivalent realizations, $\projrs_k(\coty)$ is homeomorphic to $\projrs_1(\omega) \times \projrs_1(\psi_1) \times \dots \times \projrs_1(\psi_k)$, and since $\projrs_1(\omega)$ is contractible and $\projrs_1(\psi_i)$ is homotopic to $\projrs_1(\chi_i)$, $\projrs_k(\coty)$ is homotopic to $\projrs_1(\chi_1) \times \dots \times \projrs_1(\chi_k)$.
\qed

\section{Open problems and concluding remarks} \label{conclude}

What is the smallest integer $k_n$ such that any order type on $n$ elements can be realized by convex $k_n$-gons? 
Theorem~\ref{k-bounds} gives asymptotic bounds, but there remains a wide gap.

Does universality hold for realizations of {\em order types} by $k$-gons? The authors were able to establish the weaker result, that universality holds for {\em non-crossing} arrangements of $k$-gons. That is, arrangements for which every {\em pair} of bodies has the same combinatorial type as a pair of distinct points, instead of every triple. The proof is similar to that of Theorem~\ref{realizations}, but depends on a construction where certain pairs bodies of the arrangement are disjoint in every realization by $k$-gons; see \fig~\ref{hidde}. 

\begin{figure}[h] 
\centering
\begin{tikzpicture}[scale=.7] 
\begin{scope}[xshift = -9cm, scale = .5]
\fill[black]
(152:2cm) circle [radius = .08]
(332:2cm) circle [radius = .08];
\draw[blue] 
(72:3cm)--(232:3cm)
(252:3cm)--(52:3cm);
\end{scope}

\begin{scope}[xshift = -4cm, scale = .5]
\fill[blue!15!white, opacity= .4] 
(0:5cm) \foreach \x in {0, 72 ,...,170} {-- (\x:5cm)  } --cycle ;
\draw[blue] 
(0:5cm) \foreach \x in {0, 72 ,...,170} {-- (\x:5cm)  } --cycle ;
\fill[blue!15!white, opacity= .4] (180:5cm) \foreach \x in {180, 252
  ,...,359} {-- (\x:5cm)  } --cycle ;
\draw[blue] (180:5cm) \foreach \x in {180, 252 ,...,359} {-- (\x:5cm)
} --cycle ; 
\fill[black]  \foreach \x in {36, 108 ,...,120} {
    -- (\x:4.5cm) circle [radius = .08]  } ;
\fill[black]  \foreach \x in {216, 288 ,...,359} {
    -- (\x:4.5cm) circle [radius = .08]  } ;
\fill[violet!15!white, opacity=.4] 
(82:4.15cm)--(162:4.5cm)--(242:4.15cm) --cycle 
(262:4.15cm)--(342:4.5cm)--(62:4.15cm) --cycle;
\draw[violet] 
(82:4.15cm)--(162:4.5cm)--(242:4.15cm) --cycle 
(262:4.15cm)--(342:4.5cm)--(62:4.15cm) --cycle;
\end{scope}

\begin{scope}[xshift = 4cm,  scale = .83, rotate = 25]
\fill[blue!15!white, opacity = .4]  (0:5cm) \foreach \x in {0, 50,
  ..., 177} {-- (\x:5cm)   } --cycle ;
\draw[blue]  (0:5cm)\foreach \x in {0, 50, ..., 177} {-- (\x:5cm)   }
--cycle ; 
\fill[blue!15!white, opacity=.4] (180:5cm) \foreach \x in {180, 230,
  ..., 356} {-- (\x:5cm) } --cycle ;
\draw[blue] (180:5cm) \foreach \x in {180, 230, ..., 356} {-- (\x:5cm)
} --cycle ; 
\fill[violet!15!white, opacity = .4]
(55:4.7cm) --(100:4.8cm) --(165:4.72cm) --(225:4.7cm) --cycle ;
\draw[violet]
(55:4.7cm) --(100:4.8cm) --(165:4.72cm) --(225:4.7cm) --cycle ;
\fill[violet!15!white, opacity = .4]
(235:4.7cm) --(280:4.8cm) --(345:4.72cm) --(45:4.7cm) --cycle;
\draw[violet]
(235:4.7cm) --(280:4.8cm) --(345:4.72cm) --(45:4.7cm) --cycle;
\fill[teal!15!white, opacity = .4]
(103:4.45cm) -- 
(165:4.57cm) --   
(230:4.655cm) --   
(277:4.45cm) --cycle
(283:4.45cm) --   
(345:4.57cm) --
(50:4.655cm) --
(97:4.45cm) --cycle;
\draw[teal]
(103:4.45cm) -- 
(165:4.57cm) --   
(230:4.655cm) --   
(277:4.45cm) --cycle
(283:4.45cm) --   
(345:4.57cm) --
(50:4.655cm) --
(97:4.45cm) --cycle;
\fill[black]  
\foreach \x in {25, 75, ..., 150} 
{-- (\x:4.73cm) circle [radius = .05]  } 
\foreach \x in {205, 255, ..., 345} 
{-- (\x:4.73cm) circle [radius = .05]  } ;
\end{scope}

\end{tikzpicture}
\caption{\ft Construction for $k=2,3,4$ from the proof of universality for non-crossing combinatorial types.}
\label{hidde}
\end{figure}
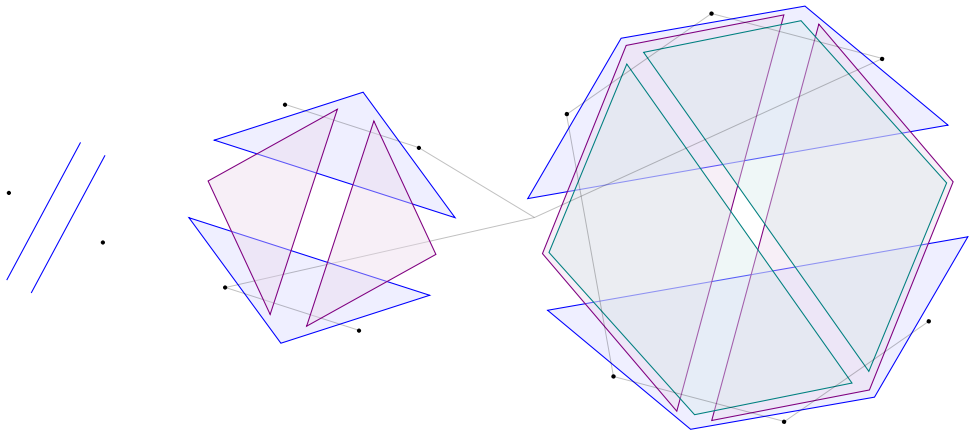

\FloatBarrier

Is the realization space of $k$-gons with the combinatorial type of $n$ points in convex position contractible?  While this may be the simplest order type, difficulties arise when the $k$-gons intersect; see \fig~\ref{windmill}.

\begin{figure}[h]
\centering
\begin{tikzpicture}[x={(-30:.6cm)},y={(60:.6cm)}]
\draw[thick,violet,fill=violet,fill opacity=0.03]
(-1.5,-.4) -- (-.5,1.8) -- (4.34,-.4) --cycle
;
\draw[thick,blue,fill=blue,fill opacity=0.03]
(.4,-1.5) -- (-1.8,-.5) -- (.4,4.34) --cycle
;
\draw[thick,black,fill=black,fill opacity=0.03]
(1.5,.4) -- (.5,-1.8) -- (-4.34,.4) --cycle
;
\draw[thick,cyan,fill=cyan,,fill opacity=0.03]
(-.4,1.5) -- (1.8,.5) -- (-.4,-4.34) --cycle
;
\end{tikzpicture}
\caption{\ft Four triangles having the same combinatorial type as vertices of a square.}
\label{windmill}
\end{figure}
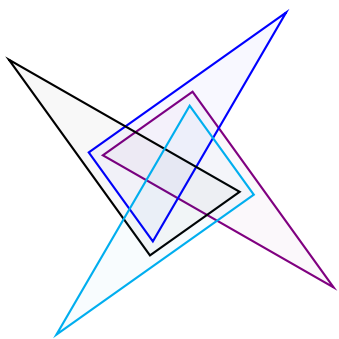

\FloatBarrier

Can every {\em order type} be realized by arrangements of {\em pairwise disjoint} bodies? Ziegler has given a construction of $2^{O(n^2)}$ distinct order types, all of which can be realized by pairwise disjoint bodies \cite[Theorem 7.4.2]{OMS}. This fails for combinatorial types in general, as there exists an arrangement of 4 {\em non-crossing} bodies in which some pair must always intersect; see \fig~\ref{notdisjoint}.  

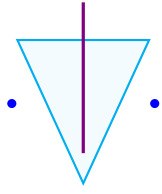
\begin{figure}[h]
\centering
\begin{tikzpicture}
\draw[thick,cyan,fill=cyan,fill opacity=0.05]
(30:1) -- (150:1) -- (270:1.4) -- cycle
;
\fill[blue]
(-20:1) circle [radius = .06]
(200:1) circle [radius = .06]
;
\draw[very thick,violet]
(0,-1) -- (0,1)
;
\end{tikzpicture}
\caption{\ft An arrangement with combinatorial type that cannot realized by disjoint bodies.}\label{notdisjoint}
\end{figure}

\FloatBarrier

Our results mostly focus on the cases where $k$ is a constant or $k$ is infinite. It would be interesting to understand how the realization space depends on $k$ as a function on $n$. For instance, Theorem \ref{realizations} states that $\projrs(\coty)$ is contractible for any non-layered combinatorial type $\coty$ on $[n]$, but is there a function $p(n)$ which guarantees that $\projrs_{p(n)}(\coty)$ is contractible? A natural guess would be $p(n) \in O(n^2)$. For combinatorial types $\coty$ on $[n]$, are there upper bounds on the Betti numbers of $\projrs_k(\coty)$ in terms of $k$ and $n$?

Unlike the order type, the definition of combinatorial type is non-local, in the sense that order type depends only on information about triples, whereas combinatorial type depends on global information.  In fact, there are pairs of arrangements where for each triple, the corresponding combinatorial type of that triple is the same for both arrangements, but the combinatorial type of the two arrangements overall is not the same; see \fig~\ref{nonlocal}. Can combinatorial type be defined from local information when the local complexity is bounded? More specifically, is there an integer $m_t$, such that the combinatorial type of an arrangement is uniquely determined by the combinatorial type of each sub-arrangement of $m_t$ bodies, provided that each pair of bodies has at most $t$ common supporting tangents? 

\begin{figure}[h] 
\centering
\begin{tikzpicture}

\begin{scope}[xscale=-1]

\draw[thick]
(30:1) -- (90:1) -- (150:1) -- (210:1) -- (270:1) -- (330:1) -- cycle;

\draw[thick, cyan]
(60:1) -- ++(0,-1.5);

\draw[thick, blue]
(120:1) -- ++(0,-1.5);

\node at (0:1) {\color{teal} $\bullet$};
\node at (180:1) {\color{violet} $\bullet$};

\end{scope}

\begin{scope}[shift={(-2.5cm,-3cm)},scale=.3]

\draw[violet, thick]
(0,4) \ds \dg{2} \de \g{0} \ds \de \g{6} \us \ug{3} \ue
;

\draw[thick]
(0,3) \us \ue \ds \du \ue \ds \du \ue \ds \du \ue \g{4} \ds \de
;

\draw[blue, thick]
(0,2) \g{1} \us \ug{1} \ud \dg{1} \de \g{0} \ds \de \g{3} \us \ug{1} \ue \ds \de \g{1}
;

\draw[cyan, thick]
(0,1) \g{2} \us \ue \g{0} \us \ug{1} \ud \dg{1} \de \g{1} \us \ue \ds \dg{1} \de \g{2}
;

\draw[teal, thick]
(0,0) \g{4} \us \ue \g{0} \us \ug{2} \ud \dg{3} \de \g{3}
;

\end{scope}

\begin{scope}[shift={(7cm,0cm)},xscale=-1]

\draw[thick]
(30:1) -- (90:1) -- (150:1) -- (210:1) -- (270:1) -- (330:1) -- cycle;

\draw[thick, cyan]
(-60:1) -- ++(0,1.5);

\draw[thick, blue]
(120:1) -- ++(0,-1.5);

\node at (0:1) {\color{teal} $\bullet$};
\node at (180:1) {\color{violet} $\bullet$};

\end{scope}

\begin{scope}[shift={(4.5cm,-3cm)},scale=.3]

\draw[violet, thick]
(0,4) \ds \dg{2} \de \g{0} \ds \de \g{6} \us \ug{3} \ue
;

\draw[thick]
(0,3) \us \ue \ds \du \ue \g{2} \ds \du \ue \ds \du \ue \g{2} \ds \de
;

\draw[blue, thick]
(0,2) \g{1} \us \ug{1} \ud \dg{2} \de \g{3} \us \ue \us \ue \ds \de \g{1}
;

\draw[cyan, thick]
(0,1) \g{2} \us \ue \us \ue \ds \de \g{1} \us \ug{1} \ud \dg{2} \de \g{2}
;

\draw[teal, thick]
(0,0) \g{4} \us \ug{3} \ud \dg{2} \de \ds \de \g{3}
;

\end{scope}

\end{tikzpicture}
\caption{\ft \textbf{Top:} Two arrangements that are related by a bijection preserving the combinatorial type of triples, but that do not have the same combinatorial type. \textbf{Bottom:} Their dual support system.}
\label{nonlocal}
\end{figure}
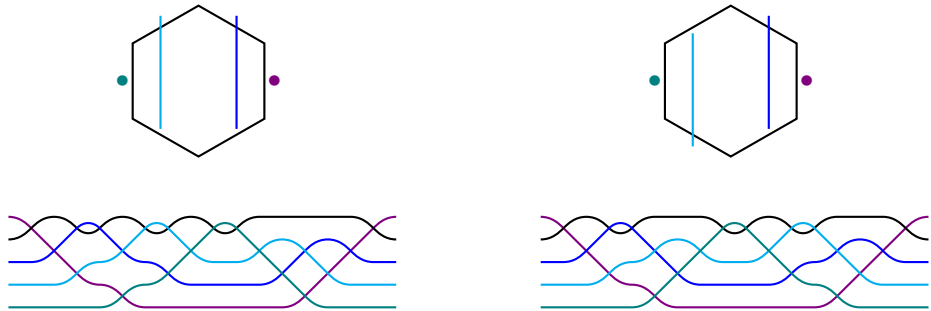

\FloatBarrier

We have bounded the complexity of the bodies by working in the space of $k$-gons. Are there other measures of complexity that yield   a universality theorem? Any continuous map from $\mathbb{R}^{c(n)}$ to the space of arrangements of convex bodies defines a finite dimensional subspace.  In the case of $k$-gons in the plane, $c(n)=2kn$.    Does universality hold for any other such map?  Consider for instance subspaces having bounded VC-dimension.

Our notion of combinatorial type extends to higher dimensions and to systems of sections of vector bundles other than the cylinder. It would be interesting to see how our results may extend in these cases.

\bibliography{realbody1}{}
\bibliographystyle{plain}

\end{document}